\documentclass[a4paper,12pt]{amsart} 
\title{Logarithmic cohomological field theories}

\usepackage{amsmath, amssymb, mathrsfs, amsthm, shorttoc, stmaryrd}
\usepackage{mathtools} 
\usepackage{hyperref}
\usepackage[usenames, dvipsnames]{xcolor}
\usepackage{subcaption}
\hypersetup{
colorlinks,
linkcolor={black},
citecolor={blue!50!black},
urlcolor={blue!80!black}
}

\usepackage[all]{xy}

\usepackage{bbm}

\usepackage{tabularx}
\usepackage{longtable}
\numberwithin{equation}{subsection}

\usepackage{enumitem}

\usepackage{cleveref}
\crefformat{equation}{(#2#1#3)}
\let\oref\ref
\let\tilde\widetilde
\AtBeginDocument{\renewcommand{\ref}[1]{\Cref{#1}}}

\usepackage{pgf,tikz}
\usetikzlibrary{matrix, calc, arrows}

\usepackage{tikz-cd} 

\makeatletter
\newcommand*{\doublerightarrow}[2]{\mathrel{
  \settowidth{\@tempdima}{$\scriptstyle#1$}
  \settowidth{\@tempdimb}{$\scriptstyle#2$}
  \ifdim\@tempdimb>\@tempdima \@tempdima=\@tempdimb\fi
  \mathop{\vcenter{
    \offinterlineskip\ialign{\hbox to\dimexpr\@tempdima+1em{##}\cr
    \rightarrowfill\cr\noalign{\kern.5ex}
    \rightarrowfill\cr}}}\limits^{\!#1}_{\!#2}}}
\newcommand*{\triplerightarrow}[1]{\mathrel{
  \settowidth{\@tempdima}{$\scriptstyle#1$}
  \mathop{\vcenter{
    \offinterlineskip\ialign{\hbox to\dimexpr\@tempdima+1em{##}\cr
    \rightarrowfill\cr\noalign{\kern.5ex}
    \rightarrowfill\cr\noalign{\kern.5ex}
    \rightarrowfill\cr}}}\limits^{\!#1}}}
\makeatother

\newcommand{\on}[1]{\operatorname{#1}}
\newcommand{\bb}[1]{{\mathbb{#1}}}

\newcommand{\ca}[1]{{\mathcal{#1}}}
\newcommand{\bd}[1]{{\mathbf{#1}}}

\newcommand{\ul}[1]{{\underline{#1}}}

\theoremstyle{definition}
\newtheorem{definition}{Definition}[section]

\theoremstyle{plain}

\newtheorem{proposition}[definition]{Proposition}
\newtheorem{lemma}[definition]{Lemma}
\newtheorem{theorem}[definition]{Theorem}
\newtheorem{corollary}[definition]{Corollary}

\newtheorem{intheorem}{Theorem}

\theoremstyle{remark}

\newtheorem{remark}[definition]{Remark}

\newtheorem{example}[definition]{Example}

\usepackage{letltxmacro}
\LetLtxMacro{\phiorig}{\phi}
\renewcommand{\phi}{\varphi}

\author{David Holmes}
\date{\today}

\address{Mathematisch Instituut\\ 
Universiteit Leiden\\
Postbus 9512\\
2300 RA Leiden\\
Netherlands}

\email{holmesdst@math.leidenuniv.nl}

\thanks{The authors are both supported by NWO grant VI.Vidi.193.006}

\author{Pim Spelier}
\date{\today}

\address{Mathematisch Instituut\\ 
Universiteit Leiden\\
Postbus 9512\\
2300 RA Leiden\\
Netherlands}

\email{spelierp@math.leidenuniv.nl}

\newcounter{nootje}
\newcommand{\beq}{\begin{equation}}
\newcommand{\eeq}{\end{equation}}
\newcommand{\beqs}{\begin{equation*}}
\newcommand{\eeqs}{\end{equation*}}

\tikzset{
  symbol/.style={
    draw=none,
    every to/.append style={
      edge node={node [sloped, allow upside down, auto=false]{$#1$}}}
  }
}

\newcommand{\ev}{\mathsf{ev}}

\DeclareMathOperator{\CH}{CH}
\DeclareMathOperator{\LogCH}{LogCH}

\DeclareMathOperator{\DR}{DR}
\DeclareMathOperator{\LogDR}{LogDR}
\DeclareMathOperator{\LLogDR}{\mathbb{L}ogDR}
\newcommand{\AJ}{\mathsf{AJ}}
\newcommand{\Picrel}{\mathsf{Pic}}

\renewcommand{\angle}[1]{\hspace{-2pt}\left\langle #1 \right\rangle}

\DeclareMathOperator{\id}{id}

\DeclareMathOperator{\Sym}{Sym}

\newcommand{\colim}{\operatornamewithlimits{colim}}

\begin{document}
\begin{abstract} 
We introduce a new logarithmic structure on the moduli stack of stable curves, admitting logarithmic gluing maps. Using this we define cohomological field theories taking values in the logarithmic Chow cohomology ring, a refinement of the usual notion of a cohomological field theory. We realise the double ramification cycle as a partial logarithmic cohomological field theory. 
\end{abstract}

\maketitle


\section*{Contents}
\contentsline {section}{\tocsection {}{1}{Introduction}}{2}{section.1}%
\contentsline {section}{\tocsection {}{2}{Conventions}}{5}{section.2}%
\contentsline {section}{\tocsection {}{3}{Log pointed curves}}{9}{section.3}%
\contentsline {section}{\tocsection {}{4}{Piecewise polynomial functions and the DR cycle}}{16}{section.4}%
\contentsline {section}{\tocsection {}{5}{Gluing log pointed curves}}{19}{section.5}%
\contentsline {section}{\tocsection {}{6}{Log CohFTs}}{22}{section.6}%
\contentsline {section}{\tocsection {}{7}{The log Double ramification cycle as a partial log CohFT}}{25}{section.7}%
\contentsline {section}{\tocsection {Appendix}{A}{Comparison of log Chow rings with log Chow rings of Barrott}}{33}{appendix.A}%
\contentsline {section}{\tocsection {Appendix}{B}{Tropical gluing}}{36}{appendix.B}%
\contentsline {section}{\tocsection {Appendix}{}{References}}{38}{section*.3}%


\newcommand{\Mtildes}{ \widetilde{\ca M}^\Sigma}
\newcommand{\sch}[1]{\textcolor{blue}{#1}}

\newcommand{\Mbar}{\overline{\ca M}}
\newcommand{\Mfrak}{\mathfrak{M}}
\newcommand{\Jac}{\mathfrak{Jac}}
\newcommand{\J}{\mathsf{J}}

\newcommand{\isom}{\stackrel{\sim}{\longrightarrow}}
\newcommand{\Ann}[1]{\on{Ann}(#1)}
\newcommand{\field}{K}
\newcommand{\cat}[1]{\bd{#1}}
\renewcommand{\log}{{\mathsf {log}}}
\newcommand{\trop}{{\mathsf {trop}}}
\newcommand{\op}{{\mathsf {op}}}
\newcommand{\M}{{\mathsf {M}}}
\newcommand{\gp}{{\mathsf {gp}}}
\newcommand{\punc}{{\mathsf {punc}}}
\newcommand{\st}{{\mathsf {st}}}
\newcommand{\gl}{{\mathsf{gl}}}
\newcommand{\glue}{{\mathsf{gl}}}
\newcommand{\et}{{\mathsf{et}}}
\newcommand{\vfc}{{\mathsf{vfc}}}
\newcommand{\CHop}{{\mathsf{CH}_\op}}
\newcommand{\sPP}{{\mathsf{sPP}}}
\newcommand{\ghost}{\overline{{\mathsf {M}}}}
\newcommand{\Gmlog}{\bb G_m^\log}

\newcommand{\Mpt}{\bb M}
\newcommand{\Mptst}{\Mpt^\st}
\newcommand{\Mst}{\Mbar}
\newcommand{\Mptststr}{\Mpt^{\st,\text{str}}}
\newcommand{\Cpt}{\bb C}
\newcommand{\Cptst}{\Cpt^\st}
\newcommand{\Divpt}{\bb{D}\mathsf{iv}}
\newcommand{\Div}{\mathsf{Div}}

\newcommand{\tensor}{\otimes}
\newcommand{\bun}{{\text{bun}}}
\newcommand{\Ad}{\operatorname{Ad}}
\newcommand{\Spec}{\operatorname{Spec}}
\newcommand{\Sch}{\operatorname{Sch}}
\newcommand{\LogSch}{\operatorname{LogSch}}

\newcommand{\GL}{\operatorname{GL}}
\newcommand{\SL}{\operatorname{SL}}

\newcommand{\Hom}{\operatorname{Hom}}
\newcommand{\Aut}{\operatorname{Aut}}

\newcommand{\RCH}{\operatorname{RCH}}
\newcommand{\RLogCH}{\operatorname{RLogCH}}
\newcommand{\PP}{\operatorname{PP}}
\newcommand{\im}{\operatorname{im}}

\newcommand{\RPC}{\operatorname{RPC}}
\newcommand{\RPCC}{\operatorname{RPCC}}

\newcommand{\A}{\mathbb{A}}
\newcommand{\B}{\mathsf{B}}
\newcommand{\C}{\mathbb{C}}
\newcommand{\D}{\mathbb{D}}
\newcommand{\E}{\mathbb{E}}
\newcommand{\F}{\mathbb{F}}
\newcommand{\G}{\mathbb{G}}
\newcommand{\I}{\mathbb{I}}
\newcommand{\K}{\mathbb{K}}
\newcommand{\N}{\mathbb{N}}
\renewcommand{\P}{\mathbb{P}}
\newcommand{\Q}{\mathbb{Q}}
\newcommand{\R}{\mathbb{R}}
\newcommand{\T}{\mathbb{T}}
\newcommand{\U}{\mathbb{U}}
\newcommand{\V}{\mathbb{V}}
\newcommand{\W}{\mathbb{W}}
\newcommand{\X}{\mathbb{X}}
\newcommand{\Y}{\mathbb{Y}}
\newcommand{\Z}{\mathbb{Z}}
\newcommand{\Acal}{\mathcal{A}}
\newcommand{\Bcal}{\mathcal{B}}
\newcommand{\Ccal}{\mathcal{C}}
\newcommand{\Dcal}{\mathcal{D}}
\newcommand{\Ecal}{\mathcal{E}}
\newcommand{\Fcal}{\mathcal{F}}
\newcommand{\Gcal}{\mathcal{G}}
\newcommand{\Hcal}{\mathcal{H}}
\newcommand{\Ical}{\mathcal{I}}
\newcommand{\Jcal}{\mathcal{J}}
\newcommand{\Kcal}{\mathcal{K}}
\newcommand{\Lcal}{\mathcal{L}}
\newcommand{\Mcal}{\mathcal{M}}
\newcommand{\Ncal}{\mathcal{N}}
\newcommand{\Ocal}{\mathcal{O}}
\newcommand{\Pcal}{\mathcal{P}}
\newcommand{\Qcal}{\mathcal{Q}}
\newcommand{\Rcal}{\mathcal{R}}
\newcommand{\Scal}{\mathcal{S}}
\newcommand{\Tcal}{\mathcal{T}}
\newcommand{\Ucal}{\mathcal{U}}
\newcommand{\Vcal}{\mathcal{V}}
\newcommand{\Wcal}{\mathcal{W}}
\newcommand{\Xcal}{\mathcal{X}}
\newcommand{\Ycal}{\mathcal{Y}}
\newcommand{\Zcal}{\mathcal{Z}}

\newcommand{\dual}{\wedge}

\newcommand{\ol}{\overline}

\newcommand{\str}{\text{str}}
\newcommand{\pt}{\{*\}}

\newcommand{\pie}{\mathring}

\section{Introduction}


\subsection{Background}
Cohomology classes on the moduli space of curves are the central object of study in Gromov-Witten theory. One of the most important structures such cohomology classes can carry is that of a \emph{cohomological field theory} (CohFT), introduced in the 1990s by Kontsevich and Manin \cite{Kontsevich1994Gromov-Witten-c} to capture the formal properties of the virtual fundamental class in Gromov-Witten theory.  A collection of classes forms a CohFT if they are compatible with pulling back along the \emph{gluing maps}
\begin{align*}
\Mbar_{g_1,n_1+1} \times \Mbar_{g_2,n_2+1} & \to \Mbar_{g_1 + g_2, n_1 + n_2}, \\
\Mbar_{g-1,n+2} & \to \Mbar_{g,n}. 
\end{align*}
(see \cite{Pandharipande2018CohFTCalculations} for a precise definition). CohFT structures have allowed for the computation of many interesting tautological classes (see e.g. \cite{Pandharipande2015RelationsOnMgnVia3spin,Pandharipande2018CohFTCalculations}), often via the Givental--Teleman reconstruction of semi\-simple CohFTs, and CohFTs form a bridge between algebraic geometry and integrable hierarchies (see e.g. \cite{Dubrovin2001NormalFormsOfHierarchies,Buryak2019Quadratic-doubl}). The main example of a cohomological field theory is the collection of Gromov--Witten invariants of a space $X$. For some simple spaces $X$ the Givental--Teleman reconstruction even recovers all invariants of $X$ from the $g = 0, n = 3$ invariants.

\subsection{Logarithmic classes and logarithmic gluing maps}
In recent years enhancing Gromov-Witten theory with additional data coming from logarithmic (log) geometry \cite{chen2014stable,Abramovich2014Stable-logaritmic-maps-II,Gross2013Logarithmic-gro} is becoming increasingly important; this allows one to capture \emph{tangency conditions} (recent examples include \cite{BarrottNabijou,Ranganathan2023Logarithmic-Gromov-Witten,van2023stable}), and plays a key role in \emph{degeneration arguments} \cite{abramovich2020decomposition,kim2018degeneration,ranganathan2019logarithmic}. Log Gromov-Witten invariants live most naturally in the \emph{log Chow ring} of the moduli space of (stable) log curves; this is the colimit over log blowups of the moduli space (equivalently, iterated blowups in boundary strata), see \ref{subsec:logchow} for a precise definition of this ring, and  \cite{Molcho2021The-Hodge-bundl,Holmes2022Logarithmic-double,Molcho2021A-case-study-of,Holmes2021Logarithmic-int} for examples and applications.

So far there is no notion of a cohomological field theory for classes in the log Chow ring, because there is no simple theory of log gluing maps between moduli spaces of log curves, and hence there exists no pullback along the gluing map on log Chow. We give a refinement of the moduli space of log curves, which admits gluing maps. This allows us to define a `logarithmic cohomological field theory' and we give some examples.

We show additionally that the gluing of two curves satisfies the expected universal property for maps to a space $X$ (\ref{thm:universalpropertygluing}). This raises the question of how logarithmic Gromov--Witten invariants pull back along the log gluing maps. We show that the log double ramification cycle, roughly counting maps to $\P^1$ with specified ramification over $0$ and $\infty$, is a \emph{partial} logarithmic cohomological field theory. This means that (among other axioms) the pullback along the separating gluing map $\Mbar_{g_1,n_1+1} \times \Mbar_{g_2,n_2+1}  \to \Mbar_{g_1 + g_2, n_1 + n_2}$ admits a recursive formula. However, it does not satisfy the loop axiom, the recursive formula for the pull-back along the loop gluing map $\Mbar_{g-1,n+2} \to \Mbar_{g,n}$ required for a log CohFT. In \cite{Spelier2025SplittingFormulaLogDR} the second author uses the results from this paper to prove a replacement of the loop axiom: this is still a recursive formula, but takes a slightly different form. In our later work \cite{herr2026splitting} joint with Leo Herr we compute the pullbacks along the gluing maps of any log Gromov--Witten invariant. 

To refine log curves in order to obtain gluing maps, we make a simple change: we require that the marked points of the log curves are \emph{logarithmic} sections, rather than just sections on the underlying schemes (see \ref{def:logpointedcurve} for details). We call these \emph{log pointed} curves, and write $\Mpt^\st_{g,n}$ for the stack of stable log pointed curves. 

\begin{intheorem}\label{inthm:gluing}
The stack $\Mpt^\st_{g,n}$ of stable log pointed curves is an algebraic stack with log structure. It admits a forgetful log map to the stack $\Mbar_{g,n}$ of stable curves. This map is an isomorphism on underlying algebraic stacks - it only changes the log structure. There are log gluing maps 
\begin{align*}
\gl\colon \Mptst_{g_1,n_1+1} \times \Mptst_{g_2,n_2+1} & \to \Mptst_{g_1 + g_2, n_1 + n_2}, \\
\gl\colon \Mptst_{g-1,n+2} & \to \Mptst_{g,n}.
\end{align*}
compatible with the usual gluing maps on $\Mbar_{g,n}$. 
\end{intheorem}

We define these gluing maps through defining \emph{pierced} curves associated to pointed curves. Like pointed curves, these have logarithmic sections, but slightly different log structure at the marked nodes. Tropically, these correspond to metrised graphs where the legs are of finite length, with the section being the endpoint of the leg. Gluing then corresponds to adding the lengths of the two legs that are being glued. See \ref{fig:tropglue} for a picture; in appendix B we show that tropical gluing is really gluing the abstract tropical curves.

\begin{figure}[ht]
\begin{center}
\begin{tikzpicture}[thick,scale=0.45,
    vert/.style={circle, draw, fill=black, inner sep=0pt, minimum width=4pt},
    endpt/.style={circle, draw, fill=white, inner sep=0pt, minimum width=4pt}]

  \node at (-3,0) {\scalebox{1}[1.25]{$\Bigg($}};

  \begin{scope}[shift={(0,0)}]
    \node[vert] (Av) at (0,0) {};
    \draw (Av) .. controls (-3,1.6) and (-3,-1.6) .. (Av);
    \draw (Av) -- node[above]{$\ell_1$} (2.5,0);
  \end{scope}

  \node at (3,0) {,};

  \begin{scope}[shift={(6,0)}]
    \node[vert] (Bv) at (0,0) {};
    \draw (Bv) -- node[above]{$\ell_2$} (-2.5,0);
    \draw (Bv) -- (2,1.5);
    \draw (Bv) -- (2,-1.5);
  \end{scope}

  \node at (8.9,0) {\scalebox{1}[1.25]{$\Bigg)$}};

  \node at (10.7,0) {$\longmapsto$};

  \begin{scope}[shift={(14.5,0)}]
    \node[vert] (Cv1) at (0,0) {};
    \node[vert] (Cv2) at (5,0) {};
    \draw (Cv1) .. controls (-3,1.6) and (-3,-1.6) .. (Cv1);
    \draw (Cv1) -- node[above]{$\ell_1 + \ell_2$} (Cv2);
    \draw (Cv2) -- (7,1.5);
    \draw (Cv2) -- (7,-1.5);
  \end{scope}

\end{tikzpicture}
\end{center}
\caption{Tropical gluing of two pierced curves: the two legs of finite length $\ell_1$ and $\ell_2$ are glued to a single edge of length $\ell_1 + \ell_2$.}
\label{fig:tropglue}
\end{figure}

Logarithmically, the gluing of two pierced curves can be constructed as a coproduct. This coproduct construction leads to the glued curve having the correct universal property (\ref{thm:universalpropertygluing}), and hence one also obtains gluing maps for moduli spaces of maps.

\subsection{The log double ramification cycle and gluing maps}
Given a vector of integers $\ul a = (a_1, \dots, a_n)$ summing to zero, the double ramification cycle $\DR_g(\ul a)$ on $\Mbar_{g,n}$ measures the locus where the divisor $\sum_i a_i p_i$ is linearly equivalent to $0$ (here $p_1,\dots,p_n$ are the $n$ sections of the curve). Formal definitions were given in \cite{Graber2003Hodge-integrals}, \cite{Holmes2017Extending-the-d}, \cite{Marcus2017Logarithmic-com}, via the Gromov-Witten theory of rubber maps, using birational geometry, and using log geometry, respectively. These equivalent constructions in fact naturally yield a slightly more refined object, the log double ramification cycle, which lies on a blowup of $\Mbar_{g,n}$ (the usual double ramification cycle is obtained by pushing forward to $\Mbar_{g,n}$). This blowup is not canonical, but the resulting class is independent of the choice in the sense that the pullbacks to any common refinement will coincide. In \cite[\S 9]{Holmes2017Multiplicativit} this was used to define the log double ramification cycle $\LogDR_g(\ul a)$ in the \emph{log Chow ring} 
\begin{equation}
\LogCH_{\text{HPS}}(\Mbar_{g,n}) = \colim_{Y \in \ca Bl(\Mbar_{g,n}) } \CH^*(Y), 
\end{equation}
where the colimit runs over smooth log blowups of $\Mbar_{g,n}$. The log double ramification cycle is shown to be tautological in \cite{Molcho2021A-case-study-of} and \cite{Holmes2021Logarithmic-int}, and an explicit formula is given in \cite{Holmes2022Logarithmic-double}, inspired by Pixton's formula as proven in \cite{Janda2016Double-ramifica}. The logarithmic double ramification cycle captures much more enumerative information than the ordinary double ramification cycle. For example, the log Gromov--Witten invariants of any toric variety can be expressed in terms of the log double ramification cycle \cite{Ranganathan2023Logarithmic-Gromov-Witten}, and Hurwitz numbers can be expressed in terms of the log double ramification cycle \cite{DoubleHurwitz2025CMR,cavalieri2024kleakydoublehurwitzdescendants}, and the formula is even used to study $\lambda_g$ \cite{Molcho2021The-Hodge-bundl}.

In \cite{Buryak2019Quadratic-doubl} it is proven that the double ramification cycle forms a partial cohomological field theory (a cohomological field theory that fails to satisfy the loop axiom). Our second main theorem is a generalisation of the result to the log double ramification cycle. In \ref{def:logDR_via_pullback} we construct a further refinement of $\LogDR_g(\ul a)$, to the log Chow ring of the stack of log pointed curves. Constructing the latter ring is actually slightly subtle. The stacks $\Mbar_{g,n}$ have a simple log structure; it is the divisorial log structure coming from the boundary. In particular, all their log blowups are birational, and can be dominated by smooth log blowups. In contrast, the stacks $\Mpt^\st_{g,n}$ have generically non-trivial log structure (in fact, the log structure has generic rank $n$, and the associated line bundles are obtained by pulling back the relative dualising sheaf of the curve along the corresponding marking). This means that log blowups of the spaces $\Mpt^\st_{g,n}$ can have quite complicated shapes; in particular, they need not be birational maps. A refined theory of log Chow rings that applies to such stacks is under development by Barrott \cite{Barrott2019Logarithmic-Cho}, but for the convenience of the reader we give in \ref{subsec:logchow} a simplified definition, and in \ref{sec:logchow} a comparison to Barrott's construction. 

The gluing maps of \ref{inthm:gluing} now allow us to define the notion of (partial) log CohFT: a system of classes in $\LogCH(\Mptst_{g,n})$ for every $g,n$ with $2g - 2 + n >0$, compatible with the gluing maps and the forgetful maps $\Mptst_{g,n+1} \to \Mptst_{g,n}$; see \ref{sec:log_cohFT} for the precise definition. With this theory in hand, in \ref{sec:logDR_as_CohFT} we prove that this lift of $\LogDR_g(\ul a)$ to $\LogCH(\Mptst_{g,n})$ forms a partial log CohFT. 
\begin{intheorem}[\Cref{thm:drpartialcohft}]
For $g,n$ with $2g-2+n > 0$, let $\pi\colon \Mptst_{g,n} \to \Mbar_{g,n}$ denote the forgetful map. For a sequence of integers $(a_1, \dots, a_n)$, let $\LLogDR(\ul a) \in \LogCH(\Mptst_{g,n})$ denote $\pi^* \LogDR(\ul a)$ with the convention that $\LogDR(\ul a) = 0$ if $\sum_i a_i \ne 0$. Then the collection \[(\LLogDR_{g,n})_{2g - 2 + n > 0}\] forms a partial log CohFT.
\end{intheorem}

In particular, there is a recursive formula for the pullback of the logarithmic double ramification cycle along the separating gluing map. It forms a \emph{partial} CohFT, which means that the loop axiom, a specific recursive formula for the pullback along the non-separating gluing map, does not hold. A replacement for the loop axiom is given by the second author in \cite[Theorem A]{Spelier2025SplittingFormulaLogDR}, where they give a recursive formula for the pullback of $\LLogDR_g(\ul a)$ using the universal property for gluing pierced curves.


\subsubsection{Minimal log CohFTs}
Every CohFT trivially gives rise to a log CohFT, but we give an example of a log CohFT that does not arise in this way.

Pandharipande and Zvonkine \cite{Pandharipande2019Cohomological-field} give examples of CohFTs coming from so-called \emph{minimal} classes in $\CH^*(\Mbar_{g,n})$; a minimal class is a class that vanishes under pullback along all gluing maps with target $\Mbar_{g,n}$. This CohFT is constructed so that a minimal class $\gamma$ is one of the values it takes. Their construction immediately generalises to our setting, and given a minimal class $\gamma$ in $\LogCH^*(\Mptst_{g,n})$, one obtains a log CohFT taking the value $\gamma$ for a certain input. In \ref{subsub:firstminimal} and \ref{subsub:secondminimal} we give examples of minimal classes in $\LogCH^*(\Mptst_{g,n}) \setminus \CH^*(\Mptst_{g,n})$, and hence of log CohFTs that are not CohFTs.



\subsection{Tropical analogue}
In \ref{sec:tropical} we give a tropical version of this story. In particular we give in \ref{def:tropicalpointed} a tropical analogue of log pointed curves. In \ref{sec:tropical} we define tropical gluing maps between moduli spaces of tropical pointed curves, and tropical evaluation maps. This appendix can be read independently from the rest of the paper.

\subsection{Comparison to punctured curves}
The theory of punctured log curves \cite{Abramovich2020Punctured-logar}, \cite{gross2023remarks} also deals with curves with different log structures at the nodes, in order to study negative tangencies and splitting morphisms. Our log pointed curves admit canonically the structure of punctured log curves (see \ref{subsec:punctured}). However, the stack of punctured log curves is larger and more complex than ours, and the analogue of a CohFT in this context is not clear to us. Our space $\Mptst_{g,n}$ does appear in \cite[\S5.2.2]{Abramovich2020Punctured-logar} (with notation $\widetilde{\frak M'}(\ca X/B, \tau)$), by specialising the construction there to the case where the target $\ca X$ is a point with trivial log structure. But the moduli space of maps from \emph{pierced} curves to a target $X$ is different from any moduli space of punctured maps. It is a refinement of a certain space of punctured maps, but our structures such as logarithmic sections, evaluation maps and gluing maps do not descend to the space of punctured maps (even in cases where all tangencies are non-negative).

\subsection{Comparison to framed curves}
The moduli spaces $\Mptst_{g,n}$ can be viewed as \emph{spaces of framed curves} in the sense of \cite{vaintrob2019moduli}, see in particular \cite[\S8.6.1]{bergstrom2023hyperelliptic}. In those works these spaces are constructed directly by adding log structure to $\Mbar_{g,n}$; the present work provides a modular interpretation for these spaces.

\subsection{Acknowledgements}
We are grateful to Leo Herr, Alessandro Giacchetto and Nitin Chidambaram for helpful discussions, to Mark Gross for sharing an early version of \cite{Abramovich2020Punctured-logar}, to Dan Petersen for pointing out the connection to framed curves, and to Dan Abramovich for suggesting the pierced log structures introduced in \ref{sec:sewing}, which significantly streamlined the glueing procedure. We thank the referees for their feedback. We are especially thankful to Johannes Schmitt for very careful feedback on an earlier draft. Claude was used to generate the tikz picture.

\section{Conventions}
\label{sec:conventions}
Everywhere except \ref{sec:sewing} we work in the category $\LogSch$ of fs log schemes over a fixed base log scheme. In \ref{sec:PP_and_DR}, \ref{sec:log_cohFT}, \ref{sec:logDR_as_CohFT} and \ref{sec:logchow} we take the base log scheme to be a point with trivial log structure, in order to have well-behaved Chow rings. In \ref{prop:LogDRinLogS} and \ref{subsub:secondminimal} we furthermore assume we are in characteristic $0$.

\subsection{Log stacks}
\begin{definition}
A \emph{log algebraic space (or log algebraic stack)} is an algebraic space (or algebraic stack) equipped with an (fs) log structure. If a fibred category $X/\LogSch$ is represented by a log algebraic stack, we write $\ul{X}$ for the underlying algebraic stack.
\end{definition}

\begin{definition}
A \emph{log stack} is a stack in groupoids over $\cat{LogSch}_k$ for the strict \'etale topology, admitting a log \'etale cover by an algebraic stack with log structure, and with diagonal representable by algebraic spaces with log structure. 
\end{definition}
\begin{remark}
Examples of log stacks include:
\begin{itemize}
\item
any algebraic stack with log structure (in which case the cover can be taken strict);
\item the log and tropical multiplicative groups $\bb G_\log$ and $\bb G_\trop$; these are the sheaves sending a log scheme $T$ to $\Gamma(T, \M_T^\gp)$ and $\Gamma(T, \ghost_T^\gp)$ respectively (see \ref{def:loglines}; in \cite[Definition~2.2.7.1]{Molcho2018The-logarithmic} they are denoted $\mathbf G_{\log}$ and $\overline{\mathbf G}_{\log}$);
\item the log and tropical Picard groups of Molcho and Wise \cite{Molcho2018The-logarithmic}.
\end{itemize}
\end{remark}


\subsection{Log Curves}
Following \cite{Kato2000Log-smooth-defo}, a \emph{log curve} is a morphism of log algebraic spaces $\pi\colon C \to S$ that is proper, integral, saturated, log smooth, and has geometric fibers which are reduced, connected and of pure dimension 1. In fact the reducedness of the fibers is implied by the other conditions: by \cite[Theorem~II.4.2]{Tsuji2019Saturated-morphisms}, a smooth integral morphism of fs log schemes is saturated if and only if the fibers of the underlying morphism of schemes are reduced. Kato proves that the underlying morphism $\ul C \to \ul S$ is then a family of \emph{nodal} curves, and gives an explicit \'etale-local description of the log structure: the morphism $C \to S$ is strict away from finitely many points in each fibre (the nodes, together with finitely many smooth points), and at each of these points it admits an explicit chart, see \cite[Table~1.8]{Kato2000Log-smooth-defo}. Kato also proves that the locus where $C \to S$ is not vertical is (Zariski-locally on $S$) a disjoint union of \emph{schematic} sections; however, these do not in general admit lifts to logarithmic sections. Our notion of a pointed log curve imposes such lifts. 

More formally, an \emph{$n$-pointed} log curve is a tuple 
\begin{equation}
    (\pi\colon C \to S, p_1, \dots, p_n)
\end{equation}

where $\pi\colon C \to S$ is a log curve, and $p_i\colon \ul S \to \ul C$ is a map of algebraic spaces, such that $C/S$ is vertical at a point $c \in C$ if and only if $c$ does not lie in the image of some $p_i$, and the $p_i$ have disjoint images. 

We denote the moduli space of $n$-pointed genus $g$ log curves by $\Mfrak_{g,n}$, and the substack of stable log curves by $\Mbar_{g,n}$ (an $n$-pointed log curve being \emph{stable} if the underlying $n$-marked nodal curve is stable). The underlying algebraic stacks $\ul{\Mfrak_{g,n}}$ and $\ul{\Mbar_{g,n}}$ are respectively the moduli space of $n$-pointed genus $g$ curves and the moduli space of stable $n$-pointed genus $g$ curves. 


\begin{remark}
The log structures on $\Mfrak_{g, n}$ and $\Mbar_{g, n}$ are divisorial, where the boundary divisor is the (normal crossings) locus of non-smooth curves. From this perspective it is clear that the markings $p_i$ cannot be taken to be logarithmic maps, since a map of normal crossings pairs with divisorial log structure is a log morphism if and only if the preimage of the boundary divisor is contained in the boundary divisor. We emphasise this point since for \emph{log pointed} curves we \emph{will} require that the $p_i$ are log morphisms (see \ref{def:logpointedcurve}). 
\end{remark}

\subsection{Log Chow rings}
\label{subsec:logchow}

Throughout this subsection we work over a point with trivial log structure. For a smooth DM stack $X$ locally of finite type, we write $\CH^*(X)$ for the Chow ring of algebraic cycles modulo rational equivalence, in the sense of \cite{Vistoli1989Intersection-th}. 

The log Chow ring of a smooth log smooth DM stack $X$ was introduced in \cite[\S9]{Holmes2017Multiplicativit}. There, we consider the category of smooth log blowups $Y \to X$; given two such models $f\colon Y_1 \to Y_2$, the map $f$ is a morphism between smooth stacks, hence lci, and so admits a Gysin pullback $f^!\colon \CH^*(Y_2) \to \CH^*(Y_1)$. The log Chow ring of $X$ is then defined as the corresponding (filtered) colimit in the category of rings:
\begin{equation}
\LogCH_{\text{HPS}}(X) = \colim_{Y \in \ca Bl(X) } \CH^*(Y). 
\end{equation}

For the present paper we need to work with the log Chow rings of DM stacks which are not log smooth; the standard example is a stratum properly contained in a log smooth scheme; such a stratum is never log smooth. A very general theory of such log Chow rings is under development by Barrott; see \cite{Barrott2019Logarithmic-Cho}. However, for the present paper we can make do with a much simpler definition, which we hope will make life easier for the reader. Below we present our definition, and then in \ref{sec:logchow} we explain why (under mild assumptions) our naive log Chow ring has a map to that of Barrott, so that our results will transfer automatically to his framework once the details of the latter are complete. 

We start by setting up our terminology.

\subsubsection{Chow cohomology}
\label{sec:chow_cohomology}
We work always with rational coefficients. 
\newcommand{\CHOP}{\on{CH}_{\sf{OP}}}

An Artin stack $X$ locally of finite type over $k$ admits various Chow theories; to simplify we assume $X$ equidimensional. 
\begin{enumerate}
\item
If $X$ is quasi-compact and stratified by global quotient stacks, Kresch \cite{Kresch1999Cycle-groups-fo} defines a Chow group $\CH(X)$, and if $X$ is smooth it comes with a ring structure from the intersection pairing. 
\item The operational Chow ring $\CHop(X)$ defined in \cite[\S2.2]{Bae2020Pixtons-formula} is the collection of morphisms $\CH(B) \to \CH(B)$ for $B \to X$ maps from finite-type schemes, compatible with proper pushforward, flat pullback, and refined Gysin pullback along lci morphisms. If schemes are replaced by algebraic spaces or DM stacks, we obtain canonically isomorphic rings. 
\item Bae and Schmitt \cite{Bae2022Chow-rings-I} consider a variant of (2) where the test objects $B$ are taken to be algebraic stacks of finite type stratified by global quotient stacks, and the compatibilities are with representable proper pushforward, flat pullback, and refined Gysin pullback along representable lci morphisms. We denote this ring by $\CHOP(X)$. 
\end{enumerate}
In this paper we will work mainly with $\CHOP$, but to help the reader relate this work to others in the literature we summarise known comparisons. 
\begin{enumerate}
\item Since every scheme is a stack, there is a natural restriction map $\CHOP(X) \to \CHop(X)$. 
\item If $X$ is smooth, quasi-compact, and stratified by global quotient stacks, then the intersection product defines a map $\CH(X) \to \CHop(X)$. 
\item 
If $X$ is smooth, quasi-compact, and stratified by global quotient stacks, then the map $\CH(X) \to \CHOP(X)$ furnished by the intersection product is an \emph{isomorphism}, see \cite[Theorem C.6]{Bae2022Chow-rings-I}
\item If $X$ is smooth, quasi-compact, and DM, then all these maps are isomorphisms, see \cite[Proposition 5.6]{Vistoli1989Intersection-th}. 
\item Even for $X$ smooth and quasi-compact, we do not know whether $\CHOP(X) \to \CHop(X)$ is an isomorphism. 
\end{enumerate}


\subsubsection{Log modifications}
A log blowup of a log scheme is a blowup in a monoid ideal, see \cite{Niziol2006Toric-singularities}. Key examples are toric blowups of toric varieties, and bases-changes of these along strict maps (in particular, log blowups are not in general birational maps). 

\begin{definition}
A morphism $f\colon \tilde X \to X$ of log algebraic spaces is a \emph{log modification} if there exists a log algebraic space $Y$ and log blowups $Y \to \tilde X$, $Y \to X$ making the obvious triangle commute.
\end{definition}
\begin{definition}
A morphism $f\colon \tilde X \to X$ of log stacks is a \emph{log modification} if for every log scheme $T$ and map $T \to X$, the natural map
\begin{equation}
T \times_X \tilde X \to T
\end{equation}
is a log modification of log algebraic spaces.
\end{definition}
\begin{remark}
In the above definition, as everywhere in this paper, the fibre product $T \times_X \tilde X$ is taken in the category of fs log stacks; note that the underlying algebraic stack of an fs fibre product is in general not the fibre product of the underlying algebraic stacks. 
\end{remark}

\subsubsection{Log Chow}

\begin{definition}
Let $X$ be a log stack. We say $X$ is \emph{dominable} if there exists a log modification $Y \to X$ with $Y$ a finite-type algebraic stack with log structure. 
\end{definition}

\begin{definition}
\label{def:logchow}
Let $X/\LogSch$ be a dominable log stack. We define the log Chow ring of $X$ as the colimit in rings
\[
\LogCH(X) \coloneqq \colim_{\widetilde{X} \to X} \CHOP \widetilde{X}
\]
where $\widetilde{X} \to X$ ranges over all log modifications with $\widetilde{X}$ representable by finite-type algebraic stacks with log structure. 
\end{definition}

The definition also makes sense for a non-dominable log stack, but is not morally correct (it gives $\bb Q$). 


\begin{example}
\label{ex:logchgmlog}
Consider $\G_m^\log$ (over the point with trivial log structure). This has exactly one non-trivial log blowup, namely $\P^1$. We see that $\LogCH(\G_m^\log) = \CH_*(\P^1) = \Q[h]/(h^2)$ where $h = [0]$ is the class of $0 \in \P^1$.
\end{example}


\begin{example}
If $X$ is a log algebraic stack with Artin fan $A$ (see \cite{Abramovich2016Skeletons-and-f} for an introduction to Artin fans) then $\CHOP(A)$ is the ring of strict piecewise polynomials on $X$, and $\LogCH(A)$ is the ring of piecewise polynomials on $X$; this can be taken as a definition, or see \cite[Theorem B]{Molcho2021A-case-study-of} for comparison to other definitions in the literature (\cite{Molcho2021The-Hodge-bundl}, \cite{Holmes2021Logarithmic-int}, \cite{Holmes2022Logarithmic-double}). 
\end{example}

\begin{remark}
Let $X$ be a log smooth stack of finite type. Then \cite[Definition~2.4]{Holmes2021Logarithmic-int} provides an alternative proposal $\LogCH_{\text{HS}}(X)$ for the log Chow ring. We do not know whether this is equivalent to the one we use here, because \cite{Holmes2021Logarithmic-int} work with $\CHop$ in place of $\CHOP$ (see \ref{sec:chow_cohomology}), but there is a natural map $\LogCH(X) \to \LogCH_{\text{HS}}(X)$, and all constructions of \cite{Holmes2021Logarithmic-int} can naturally be lifted to $\LogCH(X)$. 


\end{remark}

\subsubsection{Pullbacks}
\label{subsubsec:logpullback}

Recall that for any map $f\colon X \to Y$ of schemes, there is a pullback $f^*\colon \CHOP(Y) \to \CHOP(X)$. This notion extends to pullbacks for $\LogCH$, as we will define now.
\begin{definition}
Let $f\colon X \to Y$ be a morphism of dominable log stacks. Let $z \in \LogCH(Y)$, and let $\tilde{Y}/Y$ be a log blowup with $z \in \CH(\tilde{Y})$. Consider the fiber square
\[
\begin{tikzcd}
\tilde{X} \arrow[r, "\tilde{f}"]\arrow[d] \arrow[dr, phantom, "\lrcorner", very near start] &  \tilde{Y} \arrow[d]\\
X \arrow[r, "f"] & Y
\end{tikzcd}
\]
Then we define $f^* z \in \LogCH(X)$ to be $\tilde{f}^* z \in \CH(\tilde{X})$. As this is independent of the choice of $\tilde{Y}$, this defines a map
\[
f^*\colon \LogCH(Y) \to \LogCH(X).
\]
\end{definition}

\section{Log pointed curves}

\begin{definition}
\label{def:logpointedcurve}
Fix non-negative integers $g$ and $n$. A \emph{log $n$-pointed curve of genus $g$} is a tuple $(C/S, p_1, \dots, p_n)$ where 
\begin{enumerate}
\item
$C/S$ is a log curve;
\item the $p_i\colon S \to C$ are $S$-maps of log schemes, landing in the locus where $C$ is classically smooth over $S$, with disjoint images;
\item $C \to S$ is vertical precisely outside the images of the $p_i$. 
\end{enumerate}
\end{definition}

\begin{remark}
There is an important difference between a log pointed curve $(C/S, p_1, \dots, p_n)$ and a marked log curve $(C/S, \ul p_1, \dots, \ul p_n)$. For a log pointed curve, the sections $p_i$ are maps of log algebraic spaces, while for a marked log curve the sections $p_i$ are just maps of the underlying algebraic spaces.
\end{remark}

\begin{definition}
Let $(C/S, p_1, \dots, p_n)$ and $(C'/S', p'_1, \dots, p'_n)$ be two log pointed curves. A \emph{morphism} from the curve $(C'/S', p'_1, \dots, p'_n)$ to the curve $(C/S, p_1, \dots, p_n)$ consists of a morphism $f\colon S' \to S$ and an $S'$-isomorphism of log schemes 
\begin{equation}
\phi \colon C' \to C \times_S S'
\end{equation}
such that for every $i$, writing $(p_i)'\colon S' \to C \times_S S'$ for the map induced by $p_i$, we have $p_i' = (p_i)'$. 
\end{definition}

These objects and morphisms form the category $\Mpt_{g,n}$ of log $n$-pointed curves, and the forgetful morphism taking $(C/S, p_1, \dots, p_n)$ to $S$ (and $(f, \phi)$ to $f$) gives a functor to $\cat{LogSch}$. Straightforward checking yields the following lemma; here and below we write CFG for `category fibred in groupoids'.

\begin{lemma}
The forgetful functor $\Mpt_{g,n} \to \cat{LogSch}$ is a CFG, and the forgetful map $\Mpt_{g,n}/\LogSch \to \Mbar_{g,n}/\LogSch$ is a map of CFG's. \qed
\end{lemma}

\begin{remark}
\label{rem:hmupointed}
A notion of `pointed log curve' is defined in \cite{Huszar2019Clutching-and-gluing}, but this is different from both the marked log curves and the log pointed curves in the present work; in [loc.cit.] the word `pointed' refers to \emph{pointed monoids}: monoids containing an absorbing element $\infty$. They define a pointed log curve as a log curve where the edge length of inner edges of the graph is allowed to be $\infty$. This allows for the gluing of log curves by setting the length of the new edge to be $\infty$.

Note that a pointed monoid is never integral, as the absorbing element prevents cancellation; the theory of [loc.cit.] therefore leaves the category of fs log schemes. Their approach is however very natural from the perspective of the extended cone complexes of Abramovich--Caporaso--Payne \cite[\S 2]{Abramovich2015Tropicalization}: allowing edge lengths equal to $\infty$ compactifies the tropical moduli space, and the tropical gluing maps become inclusions of faces at infinity of this compactification \cite[\S 8]{Abramovich2015Tropicalization}.

Note that there is no natural map from the space of pointed log curves in the sense of \cite{Huszar2019Clutching-and-gluing} to the space of log curves. In addition, their theory does not seem to furnish evaluation maps in the sense of \ref{subsec:evaluation}. 
\end{remark}

\begin{definition}
\label{def:ellpi}
Let $C/S$ be a log pointed curve, and let $p_i: S \to C$ be one of the sections. Let $P_i = \im p_i$, with the strict log structure coming from the closed embedding $P_i \to C$. Then $\ghost_{P_i} = \ghost_{S} \oplus \ul{\N}$, with the natural map $\ghost_S \to \ghost_{P_i}$ induced from $P_i \to S$ being the inclusion. 

We also have a map $\ol{p_i^*}: \ghost_{P_i} = \ghost_S \oplus \ul{\N} \to \ghost_{S}$. We define $\ell(p_i) \in \ghost_S(S)$, the \emph{length of $p_i$}, to be the image of $1 \in \ul{\N}$ in $\ghost_{S}$.
\end{definition}

\begin{lemma}
\label{lem:ellpizero}
Let $(C/S,p_1,\dots,p_n)$ be a log pointed curve, and let the map $\alpha\colon \M_S \to \ca O_S$ denote the structure morphism. Then $\bar{\alpha}: \ghost_S \to \ca O_S/\ca O_S^\times$ sends $\ell(p_i)$ to $0$.
\end{lemma}
\begin{proof}
It is enough to prove this on strict henselian points, so assume $S = \Spec A$ for a strict henselian ring $A$. Then by \cite[Table~1.8]{Kato2000Log-smooth-defo}, locally the section $p_i$ of $C/\Spec k$ looks like
\[
\begin{tikzcd}
\M_A \arrow[d,"p_i" left, shift right]\oplus t^\N \arrow[r] & \arrow[d,"p_i" left, shift right] A[t] \\
\M_A \arrow[r]\arrow[u, shift right] & A \arrow[u, shift right]
\end{tikzcd}
\]
with the map $A[t] \to A$ being $t \mapsto 0$, and the map $\M_A \oplus t^\N \to \M_A$ being $t \mapsto \ell(p_i)$. The commutativity of this diagram then implies that the map $\M_A \to A$ sends $\ell(p_i)$ to $0$.
\end{proof}

We will show in \ref{sec:basic} that the moduli space $\Mpt_{g,n}$ is represented by the moduli space ${\mathfrak{M}}_{g,n}$ of pre-stable curves, but with a different log structure from the usual one. We do this by considering the minimal objects of the category $\Mpt_{g,n}$, in the sense of \cite{Gillam2012Logarithmic-sta}. However, the resulting log structure is fairly simple; it can be seen by embedding $\Mpt_{g,n}$ as a stratum of $\mathfrak M_{g, 2n}$, or as the zero section of a vector bundle over $\Mfrak_{g, n}$, as we now sketch for the benefit of the reader who prefers to skip the details of the proof. 

\begin{remark}\label{rem:2n_embedding}
Let $C/S$ be a genus $g$ log $n$-pointed curve. Then one can glue rational tails to the $n$ markings, and put two markings on each of the newly created tails. Furthermore, one can assign the newly created node at marking $i$ the length $\ell(p_i)$. This gives a strict locally closed embedding $\Mpt_{g,n} \to \frak M_{g,2n}$ of CFGs over $\LogSch$, identifying $\Mpt_{g,n}$ with the stratum corresponding to the dual graph shown in \ref{fig:dualgraph2n}.
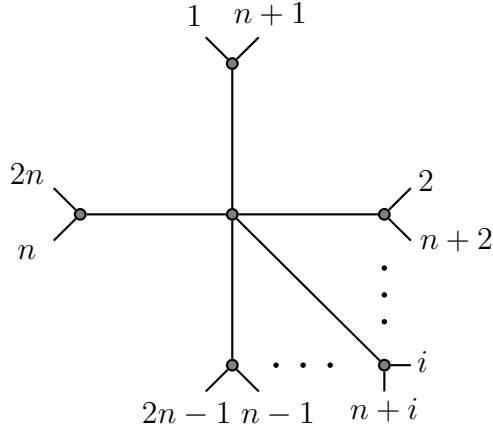
\begin{figure}[ht]
\begin{center}
\begin{tikzpicture}[thick,scale=0.5]
    \begin{scope}[every node/.style={circle, draw,fill=black!50,inner sep=0pt, minimum width=4pt}]
    \node (C) at (0,0) {};
    \node (A1) at (0,4) {};
    \node (A0) at (-4,0) {};
    \node (A3) at (4,0) {};
    \node (A4) at (0,-4) {};
    \node (A5) at (4,-4) {};
    \end{scope}

    \node[label={[shift={(0,-0.3)}]$n+1$}] (B1) at (1,5) {};
    \node[label={[shift={(0,-0.3)}]$1$}] (B2) at (-1,5) {};
    \node[label={[shift={(-.2,-0.4)}]$2n$}] (B0) at (-5,1) {};
    \node[label={[shift={(-.2,-0.4)}]$n$}] (Bm1) at (-5,-1) {};
    
    \node[label={[shift={(0.05,-0.5)}]$2$}] (B3) at (5,1) {};
    \node[label={[shift={(.45,-0.3)}]$n+2$}] (B4) at (5,-1) {};
    \node[label={[shift={(0.1,-0.6)}]$n-1$}] (B5) at (1,-5) {};
    \node[label={[shift={(-0.1,-0.6)}]$2n-1$}] (B6) at (-1,-5) {};

    \node[label={[shift={(0,-.4)}]$i$}] (B7) at (5,-4) {};
    \node[label={[shift={(-0.0,-0.55)}]$n+i$}] (B8) at (4,-5) {};

    \draw (B1) -- (A1) -- (C) -- (A0) -- (B0);
    \draw (B2) -- (A1);
    \draw (Bm1) -- (A0);
    \draw (C) -- (A3);
    \draw (C) -- (A4);
    \draw (B3) -- (A3) -- (B4);
    \draw (B5) -- (A4) -- (B6);
    \draw (C) -- (A5) -- (B7);
    \draw (A5) -- (B8);

    \path (A3) -- node[auto=false,rotate=90]{\huge{\ldots}} (A5);
    \path (A4) -- node[auto=false,rotate=0]{\huge{\ldots}} (A5);

  \end{tikzpicture}
  \caption{The dual graph of $C$}
  \label{fig:dualgraph2n}
\end{center}
\end{figure}
\end{remark}

\begin{remark}
Write $\Psi_i$ for the cotangent line bundle at the $i$th marking (so $c_1(\Psi_i) = \psi_i$), and let $V \coloneqq \Psi_1^\vee \oplus \cdots \oplus \Psi_n^\vee$ be the direct sum of the tangent line bundles at the markings, a vector bundle on $\ul{\Mfrak}_{g, n}$. We equip the total space $\on{Tot}(V)$ with the divisorial log structure whose boundary is the union of the pullback of the boundary of $\Mfrak_{g, n}$ with the $n$ coordinate hyperplanes. Equipping $\ul{\Mfrak}_{g, n}$ with the pullback log structure along the zero section $z\colon \ul{\Mfrak}_{g, n} \to \ul{\on{Tot}(V)}$ recovers precisely $\Mpt_{g,n}$. See also \ref{rem:neigbourhood}. 
\end{remark}

\subsection{Basic objects}\label{sec:basic}
Let $C/S \in \Mpt_{g,n}$ with $S$ a geometric log point, with corresponding graph $\Gamma$ and length map $\bb N^{E(\Gamma)} \to \ghost_{S}$. We also have a natural map $\N^n \to \ghost_S$, sending $i$ to $\ell(p_i)$ as in \ref{def:ellpi}. In this way we have a natural map $\bb N^{E(\Gamma)} \oplus \bb N^n \to \ghost_S$. 
\begin{definition}
We say $C/S$ is \emph{basic} if the natural map  $\bb N^{E(\Gamma)} \oplus \bb N^n \to \ghost_S$ is an isomorphism. For a general log pointed curve, we say it is \emph{basic} if it is basic at every strict geometric point.
\end{definition}

We repeat the definition of weakly terminal from \cite{Gillam2012Logarithmic-sta}.
\begin{definition}
\label{def:weaklyterminal}
Let $W$ be a category. We call a subset $P$ of the objects of $W$ \emph{weakly terminal} if
\begin{enumerate}
  \item for every object $w \in W$ there is a map $f\colon w \to p$ with $p \in P$, and
  \item \label{def:weaklyterminal:unicity} for every object $w \in W$ and every two maps $f_i \colon w \to p_i, i \in \{1,2\}$ with $p_1, p_2\in P$ there is a unique isomorphism $g: p_1 \to p_2$ with $f_2 = g \circ f_1$.
\end{enumerate}
\end{definition}

To show that the basic curves form a weakly terminal subset of $\Mpt_{g,n}$, we study the fiber of $\Mpt_{g,n}$ over a scheme.
\begin{definition}
Let $X$ be a CFG over $\LogSch$. For $\ul{S} \in \Sch$, write $X_{\ul{S}}$ for the fiber product $X \times_{\Sch} \ul{S}$. This consists of objects $(S,Y)$ where $S$ is a log scheme with underlying scheme $\ul{S}$ and $Y$ is an object in $X(S)$, and maps $(S,Y) \to (S',Y')$ that lie over $\id: \ul{S} \to \ul{S}$. 
\end{definition}

\begin{proposition}
\label{prop:basicpseudoterminal}
Let $\ul{S} \in \Sch$. Then the basic log pointed curves form a weakly terminal set inside $(\Mpt_{g,n})_{\ul{S}}$. Furthermore, the map sending a basic log pointed curve in $(\Mpt_{g,n})_{\ul{S}}$ to the unique basic marked log curve $(\Mfrak_{g,n})_{\ul{S}}$ under it is an equivalence of setoids (groupoids in which every object has trivial automorphism group).
\end{proposition}
\begin{proof}
First we note that it suffices to prove this after shrinking $\ul{S}$; the uniqueness will allow us to glue. 

Let $C/S$ be a log pointed curve with $S$ a log scheme with underlying scheme $\ul S$. Let $C_1/S_1$ be the unique minimal marked log curve under it. Then we can construct a different log curve structure $C_2/S_2$ on $\ul{C}/\ul{S}$ with sheaf of monoids $\M_{S_2} = \M_{S_1} \oplus \N^n$. To give this the structure of a log pointed curve we additionally specify that the map $p_i\colon S_2 \to \im p_i$ on the level of characteristic monoids is given by $\ghost_{S_2} \oplus \N \to \ghost_{S_2}$ sending $1 \in \N$ to $(0,e_i) \in \ghost_{S_1} \oplus \N^n = \ghost_{S_2}$. Shrinking $\ul S$, we can lift this to a map of log schemes. 

So, perhaps after shrinking $\ul S$, we get a map of log pointed curves $C/S \to C_2/S_2$, hence the log pointed curve $C/S$ has a map to a basic log pointed curve locally on the base. 

It remains to show that a map to a basic log pointed curve is unique up to unique isomorphism, as in \ref{def:weaklyterminal:unicity} of \Cref{def:weaklyterminal}. This follows immediately from the fact that, for a log scheme $(Y,\M_Y)$, the functor from the category of sharp fs monoid sheafs over $\ghost_Y$ to the category of log structures on $Y$ lying over $\M_Y$ given by $\ghost \mapsto \M_Y \times_{\ghost_Y} \ghost$ is an equivalence. 
\end{proof}

By \cite[Lemma~A.3]{Spelier2022The-moduli-stac}, a log pointed curve $C/S$ is minimal (see \cite{Gillam2012Logarithmic-sta} for the definition) if and only if lies in this weakly terminal set of basic log pointed curves inside $(\Mpt_{g,n})_{\ul{S}}$. 

\begin{corollary}
\label{cor:basiceqminimal}
An object in $\Mpt_{g,n}$ is basic if and only if it is minimal.
\end{corollary}

\begin{corollary}
The CFG $\Mpt_{g,n}/\LogSch$ satisfies the conditions of the Descent Lemma of \cite{Gillam2012Logarithmic-sta} and therefore is represented by an algebraic stack $\ul{\Mpt_{g,n}}$ with log structure. The forgetful map $\Mpt_{g,n} \to \frak M_{g,n}$ induces an isomorphism $\ul{\Mpt_{g,n}} \to \ul{\frak M_{g,n}}$ on underlying algebraic stacks.
\end{corollary}
\begin{proof}
In order to satisfy the Descent Lemma of \cite{Gillam2012Logarithmic-sta} we need the following two conditions to hold.
\begin{enumerate}
\item Every log pointed curve $C/S$ has a map to a minimal object $C'/S'$ with the induced map $\ul{S} \to \ul{S'}$ being $\id_{\ul{S}}$.
\item Let $C/S$ be a minimal log pointed curve, and let $f: S' \to S$ be a map of log schemes. Then $C_{S'}/S'$ is minimal if and only if $f$ is strict.
\end{enumerate}
By \Cref{cor:basiceqminimal}, we can instead prove these claims for the basic objects of our category.
The first condition holds by \Cref{prop:basicpseudoterminal}. To prove the second condition, we first note that being basic is a condition on strict geometric points, and hence is retained under strict pullback. Then by \Cref{prop:basicpseudoterminal} and the fact that every morphism decomposes as a strict morphism and a morphism lying over the identity of the source, the second condition holds as well.

Then by the Descent Lemma $\Mpt_{g,n}$ is represented by the stack of minimal objects $\ul{\Mpt_{g,n}}$ in $\Mpt_{g,n}$ together with its canonical log structure $\ul{\Mpt_{g,n}} \to \LogSch$ factoring through the inclusion in $\Mpt_{g,n}$.
\end{proof}

\begin{corollary}
The stack $\ul{\Mpt}_{g,n}$ is of dimension $3g-3 + n$. The stratum corresponding to a dual graph $\Gamma = (V,E)$ has generic characteristic log structure $\N^E \oplus \N^n$. 
\end{corollary}

\begin{remark}
\label{rem:mpt03}
The moduli space $\Mptst_{0,3}$ is a point with log structure $\N^3$. This log scheme has log blowups of dimension $2$, which can be explicitly constructed via the embedding in \ref{rem:2n_embedding}, by blowing up strata in $\Mbar_{0,6}$ and then base-changing to $\Mptst_{0,3}$. This implies that $\LogCH^*(\Mptst_{0,3})$ has non-trivial graded pieces of degree $1$ and $2$. In fact, $\LogCH^1(\Mptst_{0,3})$ and $\LogCH^2(\Mptst_{0,3})$ are infinite dimensional. In degree $0$, we still have $\LogCH^0(\Mptst_{0,3}) = \Q$.
\end{remark}

\begin{remark}
The stack $\Mpt_{g,n}$ is idealised log smooth over a point with trivial idealised log structure; this means that locally it can be described as a subvariety of a toric variety cut out by monomial equations. On the level of characteristic monoids we have $\ghost_{\Mpt_{g,n}} = \ghost_{\Mbar_{g,n}} \oplus \N^n$. However, it is not true that $\M_{\Mpt_{g,n}} = \M_{\Mbar_{g,n}} \oplus \N^n$ (cf. \Cref{lem:psiclasscombinatorial}).
\end{remark}

\subsection{Comparison to punctured log curves}
\label{subsec:punctured}
In this section, we will show that a log pointed curve can naturally be given the structure of a \emph{punctured} log curve. A puncturing of a log scheme $Y = (\ul{Y},\M_Y)$ is a different log scheme $Y^\punc = (\ul{Y},\M_Y^\punc)$ with the same underlying scheme, and where $\M_Y \subset \M_Y^\punc \subset \M_Y^\gp$ and $\M_Y^\punc$ satisfies certain extra conditions, see \cite[Definition~2.1]{Abramovich2020Punctured-logar}. The stack of punctured log curves $\breve{M}_{g,n}$ is constructed in \cite[2.10]{Abramovich2020Punctured-logar}.

\begin{definition}
\label{def:puncturedcurve}
Let $(C/S,p_1,\dots,p_n)$ be a log pointed curve. Let $\M$ be the verticalisation of the log structure on $C$ (the subsheaf of $\M_C$ of sections whose slopes along the markings all vanish; it agrees with $\M_C$ away from the markings), and let ${\sf P}$ be the log structure on $\ul{C}$ with respect to the divisor $\sum_i p_i$, so that $\M_C = \M \oplus_{\ca O^\times} {\sf P}$. Write $1_i$ for the section in ${\sf P}$ corresponding to $p_i$. Then we define a puncturing $C^{\punc}$ of $C$ along ${\sf P}$ to be the submonoid of $\M \oplus_{\ca O^\times} {\sf P}^\gp$ generated by $\M \oplus {\sf P}$ and $(\ell_i, -1_i)$. We define a map $\alpha\colon \M_{C^\punc} \to \ca O_C$ by sending $(\ell_i, -1_i)$ to $0$. 
\end{definition}
\begin{proof}
To check that this is a puncturing, we need to check that for every geometric point $x \in C$ we have that if $(\ell_i, -1_i)_x \in (\M_{C^\punc})_x$ is not in $(\M \oplus {\sf P})_x$ then $\alpha(\ell_i)$ is zero. This follows from \ref{lem:ellpizero}.
\end{proof}

From the definition, we immediately get the following lemma.
\begin{lemma}
\label{lem:puncgroupification}
The map $\M_{C} \to \M_{C^\punc}$ is an injective map of sheaves, and an isomorphism after groupification.
\end{lemma}

\begin{remark}
Note that there is a natural map $C^\punc \to C$, on the level of log structures near $p_i$ given by the inclusion $\M \oplus \N \to (\M \oplus \N)\angle{(\ell_i,-1)}$.
\end{remark}
\begin{definition}
Let $(C/S,p_1,\dots,p_n)$ be a log pointed curve. For every $p_i$, let $\bar{p_i}: \ghost_{C^\punc,p_i} \to \ghost_S$ denote the map given by sending $(m,n) \in (\ghost \oplus \N)\angle{(\ell_i,-1)}$ to $m + n\ell_i$.
\end{definition}

\begin{remark}
The map $\bar{p_i}$ in the previous definition is not sharp, as it sends $(\ell_i,-1)$ to $0$. Hence the map $p_i^*$ is not induced from a log section $p_i: S \to C^\punc$. Put otherwise, the log section $p_i: S \to C$ does not factor through $C^\punc \to C$. In \ref{subsec:pierced} we explore a slightly different variant of $C^\punc$ that does allow a lift of the log section $p_i$.
\end{remark}

\begin{proposition}
\Cref{def:puncturedcurve} defines a log monomorphism $\Mpt_{g,n} \to \breve{M}_{g,n}$, mapping $C/S$ to $C^{\punc}/S$.
\end{proposition}

\subsection{Evaluation maps}
\label{subsec:evaluation}

One of the major advantages of log pointed curves is the existence of evaluation maps for moduli spaces of stable maps of log curves. 

\begin{definition}
Let $X$ be a log stack, and let $\beta$ be a class of stable log maps to $X$ as in \cite[Definition~3.1]{Gross2013Logarithmic-gro}. We let $\Mptst_{g,n}(X, \beta)$ denote the moduli space with $S$-points diagrams of the form
\[
\begin{tikzcd}
C \arrow[d, "\pi"] \arrow[r, "f"] & X \\
S &
\end{tikzcd}
\]
where $(C/S, p_1, \dots, p_n)$ is a log $n$-pointed curve of genus $g$ and $f$ is a stable map such that $f$ is of class $\beta$. 

We define the evaluation map of log stacks as
\begin{align*}
\ev: \Mptst_{g,n}(X, \beta) &\to X^n \\
(C/S, p_1, \dots, p_n , f) &\mapsto (f \circ p_1, \dots, f \circ p_n). 
\end{align*}
\end{definition}

\begin{remark}
In \cite[Definitions~2.1 and~3.1]{Gross2013Logarithmic-gro} Gross and Siebert define moduli spaces of stable maps of log curves $\Mbar^\log(X, \beta)$, without log sections. The evaluation map $\Mbar^\log(X, \beta) \to X^n$ is only defined on the level of underlying schemes, and does not in general admit a lift to the level of log schemes. This makes it difficult to work with insertions from $\LogCH(X)$. 
\end{remark}

\begin{remark}
In \cite[\S 2.1]{Ranganathan2023Logarithmic-Gromov-Witten} this problem of log evaluation maps is solved for toric targets by \emph{removal of log structure on the target} (whereas we add extra log structure on the source). For a marking $i$ they study maps into a fixed stratum $W_i$ of a toric target $X$, but equip $W_i$ with its natural log structure as a toric variety, \emph{not} the log structure coming from $X$. This yields, for each $i$, a logarithmic evaluation map $\Mbar^\log(X, \beta) \to W_i$, see \cite[Proposition~2.1.2]{Ranganathan2023Logarithmic-Gromov-Witten}. 
\end{remark}

\subsection{Forgetting a marking}\label{sec:forget_marking}


We first define a map $\Mpt_{g,n+1} \to \Mpt_{g,n}$ forgetting the last marking. Let $(C/S, p_1, \dots, p_{n+1})$ be an $n+1$ log pointed curve. We define $C'/S$ and an $S$-map $\tau\colon C \to C'$ by declaring $\tau$ to be an isomorphism on the underlying curves, and on the log structures away from $p_{n+1}$; and on a neighbourhood of $p_{n+1}$ we define the log structure on $C'$ to be the kernel of the `slope' map $s_{n+1}\colon \M_{C, p_{n+1}} \to \bb N$ obtained as the composite
\begin{equation}
\M_{C, p_{n+1}}\to \ghost_{C, p_{n+1}} = \ghost_S \oplus \bb N \to \bb N. 
\end{equation}
Define $q_i = \tau \circ p_i$. Equipping $C'$ with the sections $q_1, \dots, q_n$ gives an $S$-point of $\Mpt_{g,n}$. 

Suppose now that $2g-2 + n > 0$. To build a map 
\begin{equation}
\Mptst_{g,n+1} \to \Mptst_{g,n}
\end{equation}
we must work a little harder. Let $(C/S, p_1, \dots, p_{n+1})$ be an $n+1$ stable log pointed curve, and define an $n$-log pointed curve $(C'/S, q_1, \dots, q_n)$ as before, which may not be stable. We build a stabilisation by reducing to the marked case. We write $(C'/S,\ul q_1, \dots, \ul q_n)$ for the $n$-marked log curve obtained from $(C'/S, q_1, \dots, q_n)$ by forgetting that the sections are log maps. 
\begin{lemma}
The marked log curve $(C'/S,\ul q_1, \dots, \ul q_n)$ has a stabilisation $(C^{st}/S,\ul q'_1, \dots, \ul q'_n)$ and a log map $\sigma\colon C' \to C^{st}$. 
\end{lemma}
\begin{proof}
The log structure on $\Mbar_{g,n}$ is the divisorial one coming from its boundary divisor, and the log structure on the universal curve $C_{g,n}$ is the divisorial one coming from the union of the sections with the inverse image of the boundary from $\Mbar_{g,n}$. This corresponds exactly to the boundary divisor in $\Mbar_{g,n+1}$ under the standard identification of $C_{g,n}$ with $\Mbar_{g, n+1}$. In particular, the natural isomorphism $\Mbar_{g, n+1} \to C_{g,n}$ is an isomorphism of log stacks, and yields by composition a log map $\Mbar_{g, n+1} \to \Mbar_{g, n}$. 
\end{proof}

\begin{remark}
The log structure on $C^{st}$ can also be obtained as the pushforward of the log structure on $C'$, see \cite[Appendix B]{abramovich2014comparison}; their results are stated for contractions of unmarked semistable components but extend to our setting unchanged. 
\end{remark}

We define a map
\begin{equation}
\Mptst_{g,n+1} \to \Mptst_{g,n}
\end{equation}
by sending $(C/S, p_1, \dots, p_{n+1})$ to 
\begin{equation}
(C^{st}/S, \sigma \circ q_1 , \dots, \sigma \circ q_n). 
\end{equation}

We also analyse what happens to the length of legs and edges under contraction. Fix an $n+1$ log pointed curve $(C/S, p_1, \dots, p_{n+1})$. Suppose that $S$ is atomic strictly henselian local (or, more generally, that $C/S$ is nuclear in the sense of \cite{holmes2020models}). 

Suppose first that $p_{n+1}$ is the only marking on a rational bridge, and that $\ell_1$ and $\ell_2$ are the lengths of the edges connecting the bridge to the remainder of the graph. Then in $C^{st}$ the rational bridge is contracted, and replaced by a single edge of length $\ell_1 + \ell_2$. 

The remaining case is a contracted rational tail. Suppose that, on the closed fibre, $p_{n+1}$ and $p_n$ are together on a rational tail, which carries no other markings. We write $\ell_e\in \ghost_S(S)$ for the length of the edge attaching the rational tail, and (as always) write $\ell_i$ for the length of leg $p_i$. Following the above notation, write $(C^{st}/S, q_1', \dots, q_n')$ for the stabilisation, and $\ell_i'$ for the length of $q_i'$. 

\begin{lemma}
For $1 \le i \le n-1$ we have $\ell_i = \ell_i'$, and  
\begin{equation}
\ell_n' = \ell_n + \ell_e. 
\end{equation}
\end{lemma}
\begin{proof}
The equalities $\ell_i = \ell_i'$ for $1 \le i \le n-1$ are immediate, since the stabilisation map is an isomorphism on a neighbourhood of those sections. To see what happens at $p_n$, consider a global section $\beta$ of $\ghost_{C^{st}}$ with non-zero slope at the $n$'th leg; write $\beta(v)\in \ghost_S(S)$ for the value at the generic point of irreducible component $v$ of the closed fibre, and $s$ for the outgoing slope at $q_n'$. Pulling this back from $C^{st}$ to $C$ yields a global section $\tilde \beta$ of $\ghost_{C}$. Clearly $\tilde \beta$ takes value $\beta(v)$ at generic points of irreducible components which will not be contracted. If $v_0$ is the irreducible component to which the rational tail $v_r$ is attached, the value of $\tilde \beta$ at $v_r$ is given by $\beta(v_0) + s \ell_e$. The slope of $\tilde \beta$ at $q_n'$ is still $s$.

Since $q_n'$ is constructed by composing $p_n$ with other log maps, the value $(q_n')^* \beta \in \ghost_S(S)$ is equal to the value $p_n^*\tilde \beta\in \ghost_S(S)$. From the above description, we compute
\begin{equation}
(q_n')^* \beta = \beta(v_0) + s \ell_n' \;\;\ \text{and} \;\;\; p_n^*\tilde \beta  = \tilde \beta(v_r) + s\ell_n = \beta(v_0) = s\ell_e + s\ell_n. 
\end{equation}
This immediately implies that $\ell_n' = \ell_n + \ell_e$ as required. 
\end{proof}

\section{Piecewise polynomial functions and the DR cycle}\label{sec:PP_and_DR}
Piecewise polynomials functions give an efficient way to write classes in the log Chow ring of a log stack, as exploited by \cite{Molcho2021The-Hodge-bundl,Holmes2022Logarithmic-double,Holmes2021Logarithmic-int,Molcho2021A-case-study-of}. As a first application of our log pointed curves, we show that psi classes can now also be expressed in terms of piecewise polynomial functions, and use this to write the log double ramification cycle purely in terms of piecewise polynomials. 

\subsection{Piecewise polynomial functions}
We recall the basic definitions. 
\begin{definition}
Let $X$ be an algebraic stack with log structure, locally of finite type over $k$. Write $\ca A_X$ for the Artin fan of $X$, an algebraic stack with log structure, locally of finite type over $k$. The \emph{ring of strict piecewise polynomial functions on $X$} is the Chow cohomology ring of $\ca A_X$: 
\begin{equation}
\sPP^i(X) \coloneqq \CHOP^i(\ca A_X). 
\end{equation}
The \emph{ring of piecewise polynomial functions on $X$} is the colimit of Chow cohomology rings of subdivisions of $\ca A_X$: 
\begin{equation}
\PP^i(X) \coloneqq \colim_{\tilde{\ca A} \to \ca A_X}\CHOP^i(\tilde{\ca A}). 
\end{equation}
\end{definition}
In particular, $\sPP^1(X) = \ghost_X^\gp(X) \tensor_\Z \Q$. Pulling back along the natural map $X \to \ca A_X$ gives ring homomorphisms 
\begin{equation}
\sPP^i(X) \to \CHOP^i(X)
\end{equation}
and 
\begin{equation}
\label{eq:pptologch}
\PP^i(X) \to \LogCH^i(X). 
\end{equation}

\begin{remark}
The equivalence of this definition with others in the literature (for example, in 
\cite{Holmes2021Logarithmic-int} $\sPP^i(X)$ is defined to be $\Sym^i(\ghost_S)(S)$) is proven in \cite{Holmes2022Logarithmic-double} for the case where $X$ is smooth and log smooth over a point with trivial log structure, and the general case follows from \cite[Theorem B]{Molcho2021A-case-study-of}, which is itself based on forthcoming work of Bae and Park. 
\end{remark}

\subsection{Evaluating piecewise linear functions}

Strict piecewise linear functions on $\Mpt_{g,n}$ are generated by two special classes: 
\begin{enumerate}
\item
linear functions coming from boundary divisors
\item linear functions coming from the lengths $\ell_i$ of legs. 
\end{enumerate}
In the first case the corresponding element of $\CH^1(\Mpt_{g,n})$ is the corresponding boundary divisor, just as for $\Mbar_{g,n}$. In the second case we recover psi classes, as the next lemma shows. 

\begin{lemma}
\label{lem:psiclasscombinatorial}
Let $(C/S, p_1, \dots, p_n)$ be a log-pointed curve, and for $1 \le i \le n$ let $\ell_i \in \ghost_S(S)$ be the length of marking $i$. Let $s\colon S \to \Mpt_{g,n}$ be the tautological map. Then we have an equality of operational classes\footnote{We turn the $\G_m$-torsor $O_S(-\ell_i)^\times$ into a line bundle by gluing in the $\infty$ section.} on $S$
\begin{equation}
c_1(\ca O_S(-\ell_i)) = s^*\psi_i. 
\end{equation}
\end{lemma}
\begin{proof}
Recall that $s^*\psi_i  = c_1(p_i^*\omega_{C/S})$, and by adjunction we have the equality $p_i^*\omega_{C/S} = p_i^*\ca O_C(-p_i)$. 
 It therefore suffices to show that
\begin{equation}
p_i^*\ca O_C(-p_i) = \ca O_S(-\ell_i). 
\end{equation}
We have a direct sum decomposition $\ghost_{C, p_i} = \pi^{-1} \ghost_{S} \oplus \bb N$, and $p_i$ induces a map 
\begin{equation}
p_i^*\colon \pi^{-1} \ghost_{S} \oplus \bb N \to \ghost_S
\end{equation}
which sends $(0, 1)$ to $\ell_i$. Now on a small neighbourhood $U_i$ of the image of $p_i$ we can view $(0,1)$ as a PL function on $C$, and we see that 
\begin{equation}
\ca O_C((0,1)) = \ca O_C(p_i). 
\end{equation}
Thus
\begin{equation}
p_i^*\ca O_C(-p_i) = p_i^*\ca O_C(-(0,1)) = \ca O_S(-p_i^*(0,1)) = \ca O_S(-\ell_i)
\end{equation}
as required. 
\end{proof}

\begin{remark}\label{rem:neigbourhood}
Given that the log structure on $\Mpt_{g,n}$ has generic rank $n$, one might wonder whether $\Mpt_{g,n}$ admits a log smooth map to a point with log structure $\bb N^n$. The above lemma shows that this is not in general the case. Indeed, if such a map existed then there would exist an invertible $n \times n$ integer matrix $M$ such that $M[\psi_1, \dots, \psi_n]$ is the zero vector, and this is not in general the case. 

However, a substitute can be built. Denote by $\ca Z$ the origin in the quotient $[\bb A^1/\bb G_m]$, where $\bb A^1$ is equipped with its toric log structure; so $\ca Z$ is a $B\bb G_m$ with rank 1 log structure. To give a map from a log stack $X$ to $\ca Z$ is to give a section $\bar\alpha \in \ghost_X(X)$ which maps to $0$ in $\ca O_X/\ca O_X^\times$. The lengths of the $n$ legs define a map 
\begin{equation}
\Mpt_{g,n} \to \ca Z^n, 
\end{equation}
which is easily seen to be log smooth. We might think of this as strictly embedding $\Mptst$ as the origin of the vector bundle $\Psi_1^\vee \oplus \cdots \oplus \Psi_n^\vee$ over $\Mbar_{g,n}$, where $\Psi_i$ denotes the cotangent line bundle at marking $i$. 
\end{remark}

\subsection{The double ramification cycle}

Let $\ul a \in \Z^n$ be a vector of integers summing to $0$. Then there is a locus inside ${\ca M}_{g,n}$ where the line bundle $\Ocal(a_1 p_1 + \cdots + a_n p_n)$ is trivial, called the double ramification locus. This locus has a natural extension to $\Mbar_{g,n}$, and admits a virtual fundamental class whose pushforward to $\Mbar_{g,n}$ we denote $\DR_g(\ul a) \in \CH^g(\Mbar_{g,n})$. In \cite[\S 9]{Holmes2017Multiplicativit} a natural lift $\LogDR_g(\ul a) \in \LogCH^g(\Mbar_{g,n})$ is constructed. 

\begin{definition}\label{def:logDR_via_pullback}
Let $\ul a \in \Z^n$ be a vector of integers summing to $0$. Let $\pi$ denote the forgetful map $\Mptst_{g,n} \to \Mbar_{g,n}$. Then we define the \emph{log pointed double ramification cycle} $\LLogDR_{g}(\ul a) \coloneqq \pi^*\LogDR_g(\ul a) \in \LogCH^g(\Mptst_{g,n})$.
\end{definition}
In \ref{subsec:virtualclass} we will give a more direct construction of $\LLogDR_g(\ul a)$.

\begin{proposition}
\label{prop:LogDRinLogS} Suppose the ground field $k$ has characteristic zero. Let $a \in \Z^n$ be a vector of integers summing to $0$. Then there exists a piecewise polynomial function $P$ on $\Mptst_{g,n}$ of degree $g$ whose image in $\LogCH^g(\Mptst_{g,n})$ is equal to $\LLogDR_g(\ul a)$. 
\end{proposition}
\begin{remark}
An equivalent formulation of this proposition is that this log double ramification cycle is a pullback of a cycle on the Artin fan of $\Mptst_{g,n}$; in other words, it is a purely tropical class. 
\end{remark}
\begin{proof}
We let $\Phi$ denote the map $\PP^*(\Mbar_{g,n}) \to \LogCH^*(\Mbar_{g,n})$ and $\Phi'$ the map $\PP^*(\Mptst_{g,n}) \to \LogCH^*(\Mptst_{g,n})$. By \cite[Theorem~B]{Holmes2022Logarithmic-double} and \cite[Eq.~(19)]{Holmes2022Logarithmic-double} we have the formula
\[
\LogDR_g(\ul a) = \left[\exp\left(\frac12 \left(\sum_{i = 1}^n a_i^2 \psi_i - \Phi(\Lcal)\right)\right) \cdot \Phi(\Pcal)\right]_g
\]
for certain piecewise polynomial functions $\Lcal \in \PP^1(\Mbar_{g,n})$ and $\Pcal \in \PP^*({\Mbar_{g,n}})$ dependent on $a$, where $[\cdot]_g$ denotes the codimension $g$ part.\footnote{The piecewise polynomials $\Lcal$ and $\Pcal$ depend on a choice of a stability condition $\theta$ (see \cite[Section~1.6]{Holmes2022Logarithmic-double}). The class $\LogDR_a$ is independent of this choice.} By the commutative diagram
\[
 \begin{tikzcd}
  \PP^i(\Mbar_{g,n})  \arrow[r,hookrightarrow]\arrow[d,"\Phi"] & \PP^i(\Mptst_{g,n})\arrow[d,"\Phi'"] \\
  \LogCH^i(\Mbar_{g,n}) \arrow[r,hookrightarrow] & \LogCH^i(\Mptst_{g,n})
\end{tikzcd}
\]
and \Cref{lem:psiclasscombinatorial} we see that the piecewise polynomial
\begin{equation}
P_{g,n}(\ul a) = \left[\exp\left(\frac12 \left(-\sum_{i = 1}^n a_i^2 \ell_i - \Lcal\right)\right) \cdot \Pcal \right]_g
\end{equation}
satisfies
\[
\pi^* \LogDR_g(\ul a) = \Phi'(P_{g,n}(\ul a)).\qedhere
\]
\end{proof}

\section{Gluing log pointed curves}
\label{sec:sewing}

In this section we will construct the gluing maps $\Mpt_{g_1,n_1+1} \times \Mpt_{g_2,n_2+1} \to \Mpt_{g_1+g_2,n_1+n_2}$ and $\Mpt_{g,n+2} \to \Mpt_{g+1,n}$. The analogous maps with $\Mbar$ (no log structure) in place of $\Mpt$ play a major role in the study of algebraic curves; however, when $\Mbar$ is equipped with its standard log structure no such maps exist. For log pointed curves we do have gluing maps, whose underlying maps are exactly the usual non-logarithmic gluing maps. To define the gluing maps, we first define the \emph{piercing} of a log pointed curve at a section. This is similar to, but slightly different from, the notion of the puncturing of a log pointed curve defined in \ref{subsec:punctured}

\subsection{Pierced log curves}
\label{subsec:pierced}

\begin{definition}
\label{def:piercedcurve}
Let $(\pi\colon C \to S, p_1, \dots, p_n)$ be a log pointed curve, and let $1 \le i \le n$. We define the \emph{piercing} of $C$ along $p_i$ to be the log scheme $(\pie{C}/S, p_1, \dots, p_n)$, together with map $\pie C \to C$ over $S$, defined as follows: 
\begin{enumerate}
\item
Away from $p_i$, the map $\pie C \to C$ is an isomorphism; 
\item $\pie C \to C$ is an isomorphism on underlying schemes;
\item Let $\M$ be the verticalisation of the log structure on $C$ along $p_i$, and let $\sf P$ be the log structure on $\ul C$ with respect to the divisor $p_i$, so that $\M_C = \M \oplus_{\ca O^\times} \sf P$, with natural map 
\begin{equation}
p_i^*\colon \M_{C, p_i} = \M_{p_i} \oplus_{\ca O^\times} \sf P  \to \M_S. 
\end{equation}
We define $\M_{\pie C, p_i}$ to be the largest submonoid of $\M_{p_i} \oplus_{\ca O^\times} {\sf P}^\gp$ such that the natural map 
\begin{equation}
p_i^*\colon \M_{p_i} \oplus_{\ca O^\times} \sf{P}^\gp  \to \M_S^\gp 
\end{equation}
restricts to a sharp map of log structures $\M_{\pie C, p_i} \to \M_S$; in more concrete terms, 
\begin{equation}
\label{eq:mpie}
\M_{\pie C, p_i} = \{x \in \M_{p_i} \oplus_{\ca O^\times} {\sf P}^\gp | x\in \M_{C,p_i} \vee (p_i^*x \in \M_S \wedge \alpha(p_i^* x) = 0) \}. 
\end{equation}
where $\alpha: \M_S \to \Ocal_S$ is the structure map of $S$.
\end{enumerate}

We claim that the structure map $\M_{p_i} \oplus_{\ca O^\times} {\sf P} \to \ca O_{C, p_i}$ extends uniquely to a map $\M_{\pie C, p_i} \to \ca O_{C, p_i}$. Writing $t_i \in \ca O_{C, p_i}$ for a uniformiser, there is a unique extension to a map 
\[\M_{p_i} \oplus_{\ca O^\times} \sf P^\gp \to \ca O_{C, p_i}[1/t_i]. \]
It then remains to check that the image of any element in $\M_{\pie C, p_i}\subseteq \M_{p_i} \oplus_{\ca O^\times} \sf P^\gp$ has non-negative valuation along $p_i$. This holds because any element of $\M_{\pie C, p_i}$ either lies in $\M_{c,p_i}$, or pulls back along $p_i$ to an element mapping to $0$ under the structure map.
\end{definition}

\begin{lemma}
Let $(\pi\colon C \to S, p_1, \dots, p_n)$ be as above. There is a unique lift $p_i\colon S \to \pie C$ of $p_i \colon S \to C$; on schemes this is the same map, on log structures the map $p_i^*$ is defined above.
\end{lemma}

The \emph{slope} of a PL function $m\in \ghost_{\pie C}^\gp(\pie C)$ along $p_i$ is the image of $m$ in ${\sf P}^\gp \cong \bb Z$; observe that even if $m$ lies in $\ghost_{\pie C}(\pie C)$, it can still have negative slope.

\begin{remark}
We give a short tropical interpretation of why we have this change of monoids. For ease, assume $S$ is a point and $C$ is smooth, with a single marking of length $\ell \in \ghost_S$. The tropicalisation $\Gamma$, seen as a polyhedral complex with lengths in $\ghost_S$, is a graph with one vertex and an infinite leg, but with a marked point $p$ on the leg, at distance $\ell$ from the vertex. The piecewise linear functions $\ghost_{\Gamma}^\gp$ are specified by their value at the vertex and their slope, that is $\ghost_\Gamma^\gp = \ghost_S \oplus \Z$. The monoid $\ghost_\Gamma$ consists of the elements that are non-negative on $\Gamma$, i.e. where the value on the vertex and the slope are both non-negative.

The pierced curve $\pie C$ then has tropicalistion $\pie \Gamma$ a graph with one vertex and a leg of length $\ell$. Piecewise linear functions $\ghost_{\pie \Gamma}^\gp$ are still specified by their value at the vertex and their slope, so we still have $\ghost_{\pie \Gamma}^\gp = \ghost_S \oplus \Z$. But now there are more non-negative functions; we also allow functions with negative slope that are still positive at the end point of the leg. That is,
\[
  \ghost_{\pie \Gamma} = \{f \in \ghost_S \oplus \Z : x \in \ghost_{\Gamma} \vee f(p) > 0\}.
\]

This is the tropical equivalent what \eqref{eq:mpie} says: either $x$ was already non-negative on the pointed curve, or needs to be positive at the end point of the leg (meaning it maps to $0$ in the structure sheaf).
\end{remark}

\begin{remark}
The piercing $\pie C$ of a log pointed curve $C$ at a marking $p$ is \emph{not} an fs log scheme. For example, take $C = \A^1$ with log structure given by strict inclusion of $C$ as the $x$-axis in $\A^2$, the base $\Spec k$ with log structure $\N$, and as log section $p$ the inclusion of the origin with the map on characteristic monoids $\N^2 \to \N$ given by addition. Then $\ghost_{C^\punc,p} \subset \N \times \Z$ is fs, and generated by $(0,1)$ and $(1,-1)$. But $\ghost_{\pie{C},p} \subset \N \times \Z$ consists of $(0,0)$ and the pairs $(x,y)$ with $x + y > 0$, which is not finitely generated.
\end{remark}

Note that $\pi \colon \pie C\to S$ is saturated and admits charts by integral monoids, but not charts by finitely-generated monoids (we say $\pie C\to S$ is quasi-fine and saturated). 

Since the markings $p_1, \dots, p_n$ are disjoint, we can pierce independently at any subset of $\{p_1, \dots, p_n\}$. 

\subsection{Gluing pierced log curves}\label{sec:local_pushout}

In this subsection we temporarily drop the assumption that log curves have connected fibers, in order to treat uniformly the two gluing maps above. 

Let $S$ be a log scheme and let $(C/S,p_1, \dots, p_n)$ be a log pointed curve, with $n \ge 2$. We will explicitly construct a log curve by gluing $p_1$ and $p_2$ together. For simplicity of notation, we will assume $n=2$. 

Let $(\pie C, p_1, p_2)$ be the piercing at $p_1$ and $p_2$ from \ref{def:piercedcurve}. Define $\tilde C$ to be the pushout 
\begin{equation}
 \begin{tikzcd}
  \tilde C & S \ar[l, "p"]\\
  \pie C \ar[u, "g"] & S \sqcup S \ar[l, "p_1 \sqcup p_2"] \ar[u, "i"]
\end{tikzcd}
\end{equation}
in the category of quasi-fine $S$-algebraic spaces. The underlying scheme of $\tilde C$ is the pushout of the underlying schematic diagram, and the log structure is the pullback of the corresponding diagram of log structures:
\begin{equation}
\M_{\tilde C} = g_*\M_{\pie C} \times_{(p \circ i)_* \M_{S \sqcup S}} p_*\M_S. 
\end{equation}

\begin{remark}
The pushout $\tilde C$ is \emph{not} a log curve.
The stalk of the groupification of the ghost sheaf at the new singular point $p$ is given by 
\begin{equation}
\ghost_{\tilde C,p}^\gp = \ghost_S^\gp \oplus \Z^2.
\end{equation}
\end{remark}

A section of $\M_{\tilde C, p}$ has two slopes at $p$, given by the slopes of the pullbacks along $g \circ p_1$ and $g \circ p_2$. We define $C^\gl$ to be the log scheme whose underlying scheme is that of $\tilde C$, and whose log structure is the subsheaf of $\M_{\tilde C}$ consisting of elements whose slopes at $p_1$ and $p_2$ sum to $0$. Note there is a natural map $\tilde C \to C^\gl$. 

\begin{definition}
\label{def:gluedlogcurve}
We define $C^{\gl}$, together with map of schemes $p: \ul{S} \to C^{\gl}$ mapping to the new singular point and the gluing map $\pie{C} \to C^{\gl}$, to be the \emph{gluing} of $C$ at $p_1$ and $p_2$.
\end{definition}

\begin{remark}
We shortly present an alternative definition, found and explained to us by Dan Abramovich. We let $S^{\circ \circ}$ be $\ul{S}$ with log structure $M_{S^{\circ \circ}} = \{(m,n) \in  M_{S} \oplus \Z : \bar{m} = 0 \Rightarrow n = 0\}$. This has a map to $S$ and a section $q: S \to S^{\circ \circ}$. We define the sections $p_1^\circ, p_2^\circ: S^{\circ \circ} \to \pie{C}$ to agree with $p_1,p_2$ on the level of algebraic spaces, and be given by
\[\M_{C, p_1} \to \M_{S^{\circ\circ}}\colon (a, n) \mapsto (a + n \ell_1, n)\]
and
\[\M_{C, p_2} \to \M_{S^{\circ\circ}}\colon (b, m) \mapsto (b + m \ell_2, -m)\]
on monoids.
Then one can check we have the pushout diagram
\[
 \begin{tikzcd}
  C^{\gl} & S^{\circ\circ} \ar[l]\\
  \pie C \ar[u] & S^{\circ\circ} \sqcup S^{\circ\circ} \ar[l, "p_1^\circ \sqcup p_2^\circ"] \ar[u]
\end{tikzcd}
\]
and one can define $C^\gl$ as the pushout $\pie{C} \sqcup_{S^{\circ\circ} \sqcup S^{\circ\circ}} S^{\circ\circ}$.
\end{remark}

\begin{lemma}
\label{lem:gluedlogcurve}
The prelog scheme $C^{\gl}/S$ is a log curve, with new singular point $p$ of length $\ell_1 + \ell_2$. The natural gluing map $\pie{C} \to C^{\gl}$ together with the puncturing map $\pie{C} \to C$ induces an isomorphism $C \setminus \{p_1,p_2\} \to C^{\gl} \setminus \{p\}$.
\end{lemma}
\begin{proof}
Away from $p$ this is obvious. At $p$, it follows from the computation of the fiber product
\begin{align*}
  \ghost_{C^{\gl},p} &= \left\{(a_1,s_1,a_2,s_2) \in \ghost_S \times \Z \times \ghost_S \times \Z \ \middle\vert \begin{array}{l}
    \alpha(a_1 + s_1 \ell_1) = 0 \text{ or } s_1 \geq 0  \\
    \alpha(a_2 + s_2 \ell_2) = 0 \text{ or } s_2 \geq 0  \\
    a_1 + s_1 \ell_1 = a_2 + s_2 \ell_2 \\
    s_1 + s_2 = 0
  \end{array}\right\} \\
  &= \left\{(a_1,a_2) \in \ghost_S \times \ghost_S\ \middle\vert \begin{array}{l}
    \ell_1 + \ell_2 \text{ divides } a_1 - a_2
  \end{array}\right\}. \qedhere
\end{align*}
\end{proof}

To obtain gluing maps between moduli spaces of log pointed curves, we will also need the following easy lemma.
\begin{lemma}
The construction of $C \to C^{\gl}$ commutes with base-change over $S$.
\end{lemma}

\subsection{Glueing maps}
We continue in the notation of the previous subsection.

\begin{definition}
Let $X$ be a quasi-fine log stack. We define the groupoid of pre-gluing data as the fiber product of groupoids
\[
\begin{tikzcd}
& X(\pie C) \ar[d, "{(-\circ p_1,-\circ p_2)}"] \\
X(S) \ar[r, "\Delta"] & X(S)  \times X(S). \\ 
\end{tikzcd}
\]

We say a pre-gluing datum $f: \pie C \to X$ in $X(S) \times_{X(S) \times X(S)} X(\pie{C})$ is \emph{glueable} if the slopes $\ghost_{X,\ul{f \circ p_i}} \to \Z$ at $p_1$, $p_2$ sum to zero. We denote the full subgroupoid of glueable pre-gluing data by $\Hom_{\gl}(\pie C,X)$.
\end{definition}

\begin{theorem}
\label{thm:universalpropertygluing}
Let $X$ be a log stack. Then the map $\pie C \to C^{\gl}$ induces an equivalence of groupoids
\[
\Hom(C^{\gl}, X) \isom \Hom_{\gl}(\pie C,X). 
\]
\end{theorem}
\begin{proof}
The construction of $\tilde C$ as a pushout gives an isomorphism of groupoids between $\Hom(\tilde{C},X)$ and the groupoid of pre-gluing data. The theorem follows by restricting to the full subgroupoid where the slopes sum to zero on both sides. 
\end{proof}




We will later need to glue maps for the targets $\G_m^\log$ and $\G_m^\trop$. We spell out what happens in \ref{thm:universalpropertygluing} explicitly in these two cases.

\begin{definition}
For $\beta \in \ghost_C^\gp$ resp. $M_C^\gp$, we denote by $\beta(p_i)$ the pullback along $p_i: S \to C$. Explicitly, for a piecewise linear function $\beta$ taking value $a$ at the fiberwise irreducible component containing $p_i$ and with slope $n$, the value $\beta(p_i)$ is $a + n \ell_i$.
\end{definition}

\begin{corollary}
\label{lem:gluingpiecewiselinear}
The map $\pie C \to C^\gl$ induces bijections 
\[
\ghost_{C^{\gl}}^\gp \to \{\beta \in \ghost_{C}^\gp : \beta(p_1) = \beta(p_2) \emph{ and the slopes along } p_1, p_2 \emph{ add to } 0\}
\]
and
\[
\M_{C^{\gl}}^\gp \to \{\beta \in \M_{C}^\gp : \beta(p_1) = \beta(p_2) \emph{ and the slopes along } p_1, p_2 \emph{ add to } 0\}.
\]
\end{corollary}

\subsection{Gluing maps for moduli spaces of pointed log curves}
\label{sec:sewing:sep}
By applying \ref{def:gluedlogcurve} we obtain the following.

\begin{theorem}
Fix non-negative integers $g_1$, $g_2$, $n_1$, $n_2$, $g$, $n$. Then there are gluing maps, obtained by applying the gluing construction of \ref{def:gluedlogcurve} to the universal curve:
\begin{align*}
\mathsf{gl}:\Mpt_{g_1, n_1 + 1} \times \Mpt_{g_2, n_2 + 1} &\to \Mpt_{g_1 + g_2, n_1 + n_2}\\
\mathsf{gl}:\Mpt_{g, n + 2} &\to \Mpt_{g +1, n}
\end{align*}
The gluing maps are relatively representable by log algebraic spaces. On the level of underlying algebraic stacks they coincide with the classical gluing maps. 
\end{theorem}

\section{Log CohFTs}\label{sec:log_cohFT}
With the work in \Cref{sec:sewing} on log gluing maps, the preliminary work needed for defining a log version of cohomological field theories is done. In this section we present this definition, and a few examples.

\begin{definition}
\label{def:logcft}
Consider the following data:
\begin{itemize}
\item a possibly infinite dimensional vector space $V$ with a basis $(e_i)_{i \in I}$;
\item a non-degenerate symmetric $2$-form $(\eta_{ij})_{(i,j) \in I^2}: V^{\tensor 2} \to k$ with a row and column finite inverse $(\eta^{ij})_{(i,j) \in I^2}$;
\item for every $g,n$ with $2g - 2 + n > 0$ a map
\begin{equation}
\Omega_{g,n} \colon V^{\tensor n} \to \LogCH^*(\Mptst_{g,n})
\end{equation}
\end{itemize}

If $\Omega$ satisfies the following two conditions, it is called a \emph{partial log CohFT}.
\begin{enumerate}
\item \label{def:logcft:sninvariance} The map $\Omega_{g,n}$ is equivariant with respect to the action of the symmetric group $S_n$ acting simultaneously on the source and on the target.
\item \label{def:logcft:separating} Let $g_1 + g_2 = g,n_1 + n_2 = n$, and write $\gl$ for the gluing map $\Mptst_{g_1,n_1+1} \times \Mptst_{g_1,n_2+1} \to \Mptst_{g,n}$. Then for every $(v_1,\dots,v_n) \in V^{\tensor n}$ the sum
\[
\sum_{i,j \in I} \eta^{ij} \Omega_{g_1,n_1+1}(v_1,\dots,v_{n_1}, e_i) \boxtimes \Omega_{g_2,n_2+1}(v_{n_1+1},\dots,v_{n},e_j)
\]
is a finite sum, and the resulting map
\begin{align*}
  V^{\tensor n} &\to \LogCH(\Mptst_{g_1,n_1} \times \Mptst_{g_1,n_1})\\
               (v_1,\dots,v_n) &\mapsto \sum_{i,j \in I} \eta^{ij} \Omega_{g_1,n_1}(v_1,\dots,v_{n_1}, e_i) \boxtimes \Omega_{g_2,n_2}(v_{n_1+1},\dots,v_{n},e_j)
\end{align*}
is equal to the map
\[\gl^* \circ \Omega_{g,n}: V^{\tensor n} \to \LogCH(\Mptst_{g_1,n_1+1} \times \Mptst_{g_1,n_2+1}).\]
\end{enumerate}
If furthermore the following condition, also known as the \emph{loop axiom}, is satisfied, $\Omega$ is called a \emph{log cohomological field theory}, or \emph{log CohFT} for short.
\begin{enumerate}
  \setcounter{enumi}{2}
\item \label{def:logcft:loop} Let $g > 0,n$ be integers, and write $\gl$ for the gluing map $\Mptst_{g-1,n+2} \to \Mptst_{g,n}$. Then for every $(v_1,\dots,v_n) \in V^{\tensor n}$ the sum
\[
\sum_{i,j \in I} \eta^{ij} \Omega_{g-1,n +2}(v_1,\dots,v_{n_1}, e_i, e_j)
\]
is a finite sum, and the resulting map
\begin{align*}
  V^{\tensor n} &\to \LogCH(\Mptst_{g-1,n})\\
               (v_1,\dots,v_n) &\mapsto \sum_{i,j \in I} \eta^{ij} \Omega_{g-1,n+2}(v_1,\dots,v_{n_1}, e_i, e_j)
\end{align*}
is equal to the map
\[\gl^* \circ \Omega_{g,n}: V^{\tensor n} \to \LogCH(\Mptst_{g-1,n}).\]
\end{enumerate}


Furthermore, if $\Omega$ is a (partial) log CohFT and there is a $\ul{1} \in V$ such that the following conditions hold, then $\Omega$ is a \emph{(partial) log CohFT with unit}.
\begin{enumerate}
  \setcounter{enumi}{3}
  \item \label{def:logcft:unit1} Let $\pi: \Mptst_{g,n+1} \to \Mptst_{g,n}$ be the forgetful map of \ref{sec:forget_marking}. Then \[\Omega_{g,n+1}(v_1,\dots,v_n,\ul{1}) = \pi^* \Omega_{g,n}(v_1,\dots,v_{n})\] holds for all $v_1,\dots,v_n \in V$.
  \item \label{def:logcft:unit2} The equation \[\Omega_{0,3}(v_1,v_2,\ul{1}) = \eta(v_1,v_2)\] holds for all $v_1,v_2 \in V$.
\end{enumerate}
\end{definition}

\begin{remark}
For a finite dimensional (partial) log CohFT the map $\eta$ defines an isomorphism $\eta_L\colon V \to V^*, v \mapsto \eta(v,\cdot)$. In general, the fact that $\eta$ is non-degenerate only implies that $\eta_L$ is an injection with image spanned by the duals of $\{e_i : i \in I\}$.
\end{remark}

\begin{remark}
Note that per \ref{rem:mpt03} the ring $\LogCH^*(\Mptst_{0,3})$ is not equal to $\Q$, and hence the definition of the quantum product on $V$ for CohFTs does not automatically generalise to log CohFTs. Still, given a log CohFT with unit, one can define a quantum product on $V$ by taking the quantum product of the topological part of the log CohFT. This allows one to define when a log CohFT is semisimple.

It seems interesting to ask whether semisimple log CohFTs admit a classification in the style of Givental-Teleman.
\end{remark}

\subsection{Examples of log CohFTs}
\label{subs:newCFT}
Every CohFT is naturally a log CohFT.
\begin{proposition}
Let $\Omega_{g,n}: V^{\otimes n} \to \CH^*(\Mbar_{g,n})$ be a (partial) CohFT. Then by composing with the pullback $\CH^*(\Mbar_{g,n}) \to \LogCH^*(\Mptst_{g,n})$ this defines a (partial) log CohFT.
\end{proposition}

We will now construct a large family of log CohFTs that do not come from CohFTs. We begin by recalling a construction of Pandharipande and Zvonkine \cite{Pandharipande2019Cohomological-field}. Fix $g$ and $n$, and let $\gamma \in \CH^*(\Mptst_{g,n})$ be a class that vanishes under pullback along all gluing maps. Given such a class, Pandharipande and Zvonkine explicitly construct a CohFT that for some input takes the value $\gamma$ (their definition of a CohFT is slightly different from ours, involving a $\Z_2$-grading, but by \cite[Remark~12]{Pandharipande2019Cohomological-field} in the case where $\gamma$ has even cohomological degree in $H^*(\Mptst_{g,n})$ this $\Z_2$-grading vanishes).

Suppose now that we have a class $\gamma \in \LogCH^*(\Mptst_{g,n})$ that vanish under pullback along all log gluing maps; we call such a class a \emph{minimal} class. Then the construction of \cite{Pandharipande2019Cohomological-field} immediately generalises to yield a log CohFT.

We will construct two different kinds of minimal classes.

\subsubsection{First construction of a minimal class}
\label{subsub:firstminimal}
It is immediate that any class in $\LogCH^*(\Mptst_{0,3})$ is minimal, as there are no non-trivial gluing maps\footnote{In \cite{Pandharipande2019Cohomological-field} Pandharipande and Zvonkine explicitly assume $(g,n) \neq (0,3)$. However, this is only because they have a parity condition, which is automatic in the even degree case by \cite[Remark 12]{Pandharipande2019Cohomological-field}. Their proof still works for $(g,n) = (0,3)$.}. By \ref{rem:mpt03}, the vector space $\LogCH^*(\Mptst_{0,3})$ is infinite dimensional, yielding a large number of examples of log CohFTs that are not CohFTs.

\subsubsection{Second construction of a minimal class}
\label{subsub:secondminimal}
Let $g = 3, n = 2$. Assume the characteristic of the base field is $0$. Let $\tau_0 \subset \Mpt_{g,n}^{\trop}$ be the one-dimensional cone parametrising tropical pointed curves (see \ref{def:tropicalpointed}) with one vertex, no edges, and legs of equal length. Then $\tau_0$ induces a star subdivision $\tilde{\Mpt}_{g,n}^{\trop}$ of $\Mpt_{g,n}^{\trop}$, and this induces a log blowup $\pi\colon \tilde{\Mpt}_{g,n}^\st \to \Mptst_{g,n}$. Let $\tau_1,\tau_2 \subset \Mpt_{g,n}^{\trop}$ be the two one-dimensional cones parametrising tropical pointed curves with one vertex, no edges and one of the two legs of length $0$. For $i \in \{0,1,2\}$ write $\phi_{\tau_i}$ for the unique strict piecewise linear function on $\tilde{\Mpt}_{g,n}^{\trop}$ that takes value $1$ on the primitive generator of $\tau_i$ and $0$ on all other rays.

For a piecewise polynomial function $f$ on $\Mpt_{g,n}^{\trop}$, we write $\Phi(f)$ for its image in $\LogCH(\Mptst_{g,n})$ under the map \ref{eq:pptologch}.
We claim the class $\gamma \coloneqq \lambda_{g}\lambda_{g-1} \psi_1 \Phi(\phi_{\tau_0}) \in \LogCH^7(\Mptst_{g,n})$ is minimal and not contained in $\CH^7(\Mptst_{g,n})$. For the first part, by \cite[Proposition~2.1]{buryak2016top} we have that the class $\lambda_g \lambda_{g-1} \psi_1$ is minimal and hence so is $\gamma$.

In \ref{lem:pullbackinjective} below we show that the pullback map $\CH(\Mptst_{g,n}) \to \LogCH(\Mptst_{g,n})$ is injective; hence it remains to show that $\gamma$ is not the pullback of a class in $\CH^7(\Mptst_{g,n})$. Note that $\tilde{\Mpt}_{g,n}^\st$ is of dimension $9$, while $\Mptst_{g,n}$ is of dimension $8$. In fact, $\pi$ is a $\P^1$-bundle, and in particular $\pi_*\pi^* = 0$. Note that $\Phi(\phi_{\tau_0})$ is minus the class of a cycle that maps one-to-one to $\Mptst_{g,n}$, and $\pi_* \Phi(\phi_{\tau_0}) = -[\Mptst_{g,n}]$. Hence $\pi_* \gamma = \lambda_{g}\lambda_{g-1} \psi_1 \ne 0$. So $\gamma$ is not the pullback of a class in $\Mptst_{g,n}$.

Hence there is a log CohFT $(\Omega_{g,n})_{g,n}$ for which $\Omega_{3,2}$ takes a value in $\LogCH^7(\Mptst_{3,2}) \setminus \CH^7(\Mptst_{3,2})$; in particular, it does not come by pullback from a CohFT. 

\begin{lemma}
\label{lem:pullbackinjective}
Let $X$ be a finite type algebraic stack with log structure. Then the pullback 
\begin{equation}
f\colon \CH(X) \to \LogCH(X)
\end{equation}
is injective. 
\end{lemma}
\begin{proof}
This is immediate from the fact that log blowups are proper and surjective (and both these properties are stable under base-change), together with the injectivity of pullback of Chow cohomology along proper surjective morphisms \cite[\S 4.1]{edidina2022cone}. 
\end{proof}

\section{The log Double ramification cycle as a partial log CohFT}\label{sec:logDR_as_CohFT}

We begin by lifting the construction of the double ramification cycle from \cite{Marcus2017Logarithmic-com} to our stack of log pointed curves, to give a more direct construction of the log DR cycle from \ref{def:logDR_via_pullback}. 

\begin{definition}
\label{def:loglines}
Let $S$ be a log stack. We write $\bb G_{\log, S}$ and $\bb G_{\trop, S}$ for the sheaves of abelian groups on the big strict \'etale site of $S$ given by $T \mapsto \M^\gp_T(T)$ and $T \mapsto \ghost^\gp_T(T)$ respectively. A \emph{log line} over $S$ is a $\bb G_{\log, S}$ torsor, and a \emph{tropical line} is a $\bb G_{\trop, S}$ torsor. 
\end{definition}

\begin{definition}
Given non-negative integers $g$ and $n$, and a vector of integers $\ul a = (a_1, \dots, a_n)$ with $\sum_i a_i = 0$, we define $\Divpt_{g, \ul a}$ to be the fibred category over $\Mptst_{g,n}$ whose objects are tuples 
\begin{equation}
(C/S, p_1, \dots, p_n, P/S, \bar \alpha)
\end{equation}
where $(C/S, p_1, \dots, p_n)$ is a stable log pointed curve, $P/S$ is a tropical line, and $\bar\alpha\colon C \to P$ is a map whose outgoing slope\footnote{The slope is independent of choice of local trivialisation of the torsor. } at leg $i$ is equal to $a_i$. 

We similarly define $\Divpt_{g, \ul a}(\ca O)$ to be the fibered category over $\Mptst_{g,n}$ whose objects are tuples 
\begin{equation}
(C/S, p_1, \dots, p_n, \ca P/S, \alpha)
\end{equation}
where $(C/S, p_1, \dots, p_n)$ is a stable log pointed curve, $\ca P/S$ is a \emph{log} line, and $\alpha\colon C \to \ca P$ is a map whose slope at leg $i$ is equal to $a_i$. 
\end{definition}

\begin{remark}
Locally on $S$ we can trivialise the torsors $P$ and $\ca P$, so that $\bar \alpha $ becomes a PL function on $C$, and $\alpha $ becomes a generating section of the line bundle $\ca O_C(\bar\alpha)$. We work with torsors in order that $\Divpt_{g, \ul a}$ and $\Divpt_{g, \ul a}(\ca O)$ be sheaves. 
\end{remark}

Write $\Picrel_{g,n}$ for the universal Picard space of $\Cpt_{g,n}/\Mptst_{g,n}$; this is the strict \'etale sheafification of the presheaf $T \mapsto \on{Pic}(C_T)/\on{Pic}(T)$, see \cite[\S 8]{Bosch1990Neron-models} or \cite[Part 5]{Fantechi2005Fundamental-alg} for background. It is a group algebraic space over $\Mptst_{g,n}$, which we equip with the strict log structure over $\Mptst_{g,n}$. 

There is an Abel-Jacobi map 
\begin{equation}
\AJ\colon   \Divpt_{g,  \ul a} \to \Picrel_{g,n}; \;\;\; \bar\alpha \mapsto [\ca O_C(\bar\alpha)]. 
\end{equation}
To define this map on a tuple $(C/S, p_1, \dots, p_n, P/S, \bar \alpha)$ we first shrink $S$ until we can choose an isomorphism $f\colon P \isom \bb G_{\trop, S}$, then take the line bundle $\ca O_C(f(\bar\alpha))$. The line bundle depends on the choice of trivialisation, but the class in $\Picrel_{g,n}$ does not, so this glues to a global construction. 

We let $\J_{g,n}$ denote the multidegree $0$ locus inside $\Picrel_{g,n}$ (parametrising line bundles having degree 0 on every irreducible component of every geometric fibre), often called the \emph{generalised Jacobian}.  

\begin{definition}
Define $\Divpt_{g, \ul a}^0 = \Divpt_{g,  \ul a} \times_{\Picrel_{g,n}} \J_{g,n}$.
\end{definition}

\begin{lemma}
\label{lem:key_DR_square}
The square
\begin{equation}\label{eq:key_DR_square}
 \begin{tikzcd}
  \Divpt_{g, \ul a}(\ca O) \arrow[r] \arrow[d]& \Mptst_{g,n} \arrow[d, "\ca O"]\\
  \Divpt_{g, \ul a}^0 \arrow[r, "\AJ"]& \J_{g,n}. \\
\end{tikzcd}
\end{equation}
is a pullback. 
\end{lemma}
\begin{proof}
Commutativity of the square yields a map from $\Divpt_{g, \ul a}(\ca O)$ to the fibre product. Checking that this map is an isomorphism can be done locally, and we do so by constructing an inverse map. A point of the fibre product is a point $(C/S, p_1, \dots, p_n, P/S, \bar \alpha)$ of $\Divpt_{g,\ul a}$ such that $[\ca O_C(\bar\alpha)]$ is trivial, and we may shrink $S$ until the torsor $P$ is trivial; choose an identification $P = \bb G_{\trop, S}$. Then $\bar \alpha$ is a  PL function on $C$ such that the line bundle $\ca O_C(\bar\alpha)$ descends to $S$. 

To build a point of $\Divpt_{g, \ul a}(\ca O)$, we choose $\ca P = \bb G_{\log, S}$. Perhaps shrinking $S$ again we may assume that $\ca O_C(\bar\alpha)$ is trivial (since it descends to $S$), hence we can choose a generating section $\alpha$ of $\ca O_C(\bar\alpha)$, yielding the desired point of $\Divpt_{g, \ul a}(\ca O)$. A different choice of trivialisation of $\ca O(\bar\alpha)$ would yield a different section $\alpha$, but there is a (unique) automorphism of the torsor $\ca P = \bb G_{\log, S}$ interchanging these choices. 
\end{proof}


The next lemma is the technical heart of the definition of the double ramification cycle. 
\begin{lemma}
The natural map
\begin{equation}
\pi\colon \Divpt_{g, \ul a}(\ca O) \to \Mptst_{g,n}
\end{equation}
is proper. 
\end{lemma}
\begin{proof}Properness depends only on the underlying stacks (not on the log structures). The underlying algebraic stack of $\Mptst_{g,n}$ is $\Mbar_{g,n}$, and the underlying algebraic stack of $\Divpt_{g, \ul a}$ is the same as the underlying algebraic stack of a connected component of the space $\cat{Div}$ of Marcus and Wise (and the Abel-Jacobi maps match up). The result then follows by base-changing \cite[Theorem 4.3.2]{Marcus2017Logarithmic-com}. 
%
\end{proof}

\subsection{Virtual fundamental class}
\label{subsec:virtualclass}
There are multiple possible equivalent definitions for the normal log double ramification cycle, for example \cite[Lemma~4.5]{Holmes2021Logarithmic-int} and \cite[Lemma~4.12]{Holmes2021Logarithmic-int}, or the formula \cite[Theorem~B]{Holmes2022Logarithmic-double}. Here we will use \cite[Definition~4.4]{Holmes2021Logarithmic-int} in the context of log pointed curves. We recall in outline \cite[Section~3.5]{Holmes2021Logarithmic-int}.

We start by considering the commutative diagram
\begin{equation}
\label{eq:drdiagnonlog}
\begin{tikzcd}
\Divpt_{g, \ul a}(\ca O) \arrow[r]\arrow[d] \arrow[dr, phantom, "\lrcorner", very near start] & \Mptst_{g,n} \arrow[d, "e"] \\
\Divpt_{g, \ul a}^0 \arrow[r]\arrow[d] & \J_{g,n} \\
\Mptst_{g,n}
\end{tikzcd}
\end{equation}
Here the pullback square is from \Cref{lem:key_DR_square}. 

We choose smooth log blowups $\widetilde{\Mpt}_{g,n}^\st$ resp. $\widetilde{\Divpt}_{g, \ul a}^0$ of $\Mptst_{g,n}$ and $\Divpt_{g,\ul a}^0$ such that the map $\widetilde{\Divpt}_{g,\ul a}^0 \to \widetilde{\Mpt}_{g,n}^\st$ becomes an open immersion. We define $Z$ fitting into the diagram
\begin{equation}
\label{eq:drdiaglog}
\begin{tikzcd}
Z \arrow[r]\arrow[d] \arrow[dr, phantom, "\lrcorner", very near start] \arrow[dd, bend right, swap, "j"]& \Mptst_{g,n} \arrow[d, "e"] \\
\widetilde{\Divpt}_{g, \ul a}^0 \arrow[r]\arrow[d] & \J_{g,n} \\
\widetilde{\Mpt}_{g,n}^\st \arrow[d] & \\
\Mptst_{g,n} & \\
\end{tikzcd}
\end{equation}

Let $[\widetilde{\Divpt}_{g,\ul a}^0]$ denote the fundamental class in $\CH_*(\widetilde{\Divpt}_{g, \ul a}^0)$. Then the class $j_* e^! [\widetilde{\Divpt}_{g,  \ul a}^0]$ lies in $\CH_*(\widetilde{\Mpt}^\st_{g,n})$, and taking its Poincar\'e dual\footnote{That is, allowing it to act via intersection/Gysin homomorphism. } yields a class in $\CHOP(\widetilde{\Mpt}^\st_{g,n})$. This in turn gives a class in $\LogCH^*(\Mptst_{g,n})$.
\begin{lemma}
\label{lem:llogdr}
The Poincar\'e dual of the class $j_* e^! [\widetilde{\Divpt}_{g,\ul a}^0] \in \CH_*(\widetilde{\Mpt}^\st_{g,\ul{a}})$ is equal to the class $\LLogDR(\ul a)$ from \ref{def:logDR_via_pullback}.
\end{lemma}
\begin{proof}
If the blowups $\widetilde{\Mpt}_{g,n}^\st$ respectively $\widetilde{\Divpt}_{g,\ul a}^0$ are pullbacks of blowups of $\Mst_{g,n}$ and $\Div_{g,n,\ul a}^0$, then \Cref{lem:llogdr} follows immediately from \cite[Definition~4.4]{Holmes2021Logarithmic-int}.

In \cite[Section~3.5]{Holmes2021Logarithmic-int} it is shown that there are blowups of $\Mst_{g,n}$ and $\Div_{g,n,\ul a}^0$ that satisfy the necessary conditions, so it remains to show that the class $j_* e^! [\widetilde{\Divpt_{g, \ul a}}]$ is independent of the choice of the blowups $\widetilde{\Mpt}_{g,n}^\st$ and $\widetilde{\Divpt}_{g,\ul a}^0$. As in \cite[Section~3.5]{Holmes2021Logarithmic-int}, this follows immediately from Gysin pullbacks along lci morphisms commuting with each other and with projective pushforward. 
\end{proof}

\begin{remark}
Note that the above log blowups of $\Mptst_{g,n}$ and $\Divpt_{g,  \ul a}^0$ are \emph{not} assumed to be birational. For example, if $g = 0, n = 3$ one can choose $\widetilde{\Mpt}_{g,n}^\st = \P^2$. But there are also, for any $g,n,a$, always choices of birational blowups such that $\widetilde{\Divpt}_{g,\ul a}^0 \to \widetilde{\Mpt}_{g,n}^\st$ is an open immersion.
\end{remark}

\subsection{LogDR as a partial log CohFT}
Now we can show that the classes $\LLogDR_g(\ul a)$ form a partial log CohFT.

\begin{definition}
Let $V$ be the infinite dimensional $\Q$-vector space with basis $\{e_a : a\in \Z\}$. Let $\eta\colon V \tensor V \to \Q$ be the non-degenerate symmetric form sending $e_a \tensor e_b$ to $\delta_{a+b,0}$. Write $\ul{1} = e_0 \in V$. For every $g,n$ with $2g - 2 + n > 0$, let $\Omega_{g,n}: V^{\tensor n} \to \LogCH(\Mptst_{g,n})$ be the map that sends
\[ e_{a_1} \tensor \cdots \tensor e_{a_n}
\]
to $\LLogDR(\ul a)$.
\end{definition}

In this section we will show the following theorem.

\begin{theorem}
\label{thm:drpartialcohft}
The collection $(\Omega_{g,n})_{g,n}$ forms a partial log CohFT with unit $\ul{1} \in V$.
\end{theorem}

Recall the gluing map
\[
\gl\colon \Mptst_{g_1,n_1 + 1} \times \Mptst_{g_2,n_2 + 1} \to \Mptst_{g,n}.
\]
Let $\ul{a} \in \Z^n$ be a vector summing to  $0$, and let
\[\ul b_1 = (a_1,\dots,a_{n_1}, -\sum_{i=1}^{n_1} a_i), \;\;\; \ul b_2 = (a_{n_1 + 1}, \dots, a_{n}, \sum_{i=1}^{n_1} a_i), \]
vectors of lengths $n_1 + 1$ and $n_2 + 1$ respectively, both summing to 0. We need to prove the equality
\[
\gl^* \LLogDR_{g}(\ul a) = \LLogDR_{g_1}(\ul b_1) \boxtimes \LLogDR_{g_2}(\ul b_2).
\]

To lighten notation, we write the gluing map as $\Mptst_1 \times \Mptst_2 \to \Mptst$, and  we write $\Divpt^0, \Divpt_1^0, \Divpt_2^0$ for $\Divpt_{g,a}^0, \Divpt_{g_1, b_1}^0$ and $\Divpt_{g_2,b_2}^0$ respectively, and similarly for $\Divpt(\ca O)$.

We will now construct log gluing maps for $\Divpt^0$ and $\Divpt(\ca O)$.

\begin{definition}
Define
\begin{equation}
\gl\colon \Divpt_1^0 \times \Divpt_2^0 \to \Divpt^0
\end{equation}
to be the map sending a pair 
\begin{equation}
((C_1/S, p_1, \dots, p_{n_1 + 1}, P_1, \bar\alpha_1) , (C_2, q_1, \dots, q_{n_2 + 1}, P_2, \bar\alpha_2) )
\end{equation}
to 
\begin{equation}
(\gl(C_1,C_2)/S, p_1, \dots, q_{n_2}, P_1 \otimes P_2, \bar\alpha_1 \otimes \bar\alpha_2(q_{n_2 + 1}) \vee \bar\alpha_1(p_{n_1 + 1}) \otimes \bar\alpha_2).
\end{equation}
where $\alpha_1 \otimes \alpha_2(q_{n_2 + 1}) \vee \alpha_1(p_{n_1 + 1}) \otimes \alpha_2$ denotes the gluing of functions as defined in \ref{lem:gluingpiecewiselinear}.
Similarly, define
\begin{equation}
\gl\colon \Divpt_1(\ca O) \times \Divpt_2(\ca O) \to \Divpt(\ca O)
\end{equation}
to be the map sending a pair 
\begin{equation}
((C_1/S, p_1, \dots, p_{n_1 + 1}, \ca P_1, \alpha_1) , (C_2, q_1, \dots, q_{n_2 + 1}, \ca P_2, \alpha_2) )
\end{equation}
to 
\begin{equation}
(\gl(C_1,C_2)/S, p_1, \dots, q_{n_2}, \ca P_1 \otimes \ca P_2, \alpha_1 \otimes \alpha_2(q_{n_2 + 1}) \vee \alpha_1(p_{n_1 + 1}) \otimes \alpha_2).
\end{equation}
\end{definition} 

We obtain a commutative diagram
\begin{equation}
\label{eq:bigdrdiagnonlog}
\begin{tikzcd}
& \Mfrak_1 \times \Mfrak_2 \arrow[rr] \arrow[dd, "e_1 \times e_2", near start]&  & \Mfrak \arrow[dd, "e", near start] \\
\Divpt_1(\ca O) \times \Divpt_2(\ca O) \arrow[ru]\arrow[dd] \arrow[rr, crossing over] &  & \Divpt(\ca O) \arrow[ru] &  \\
& J_1 \times J_2 \arrow[rr] & & J \\
\Divpt_1^0 \times \Divpt_2^0 \arrow[ru]\arrow[d] \arrow[rr] &  & \Divpt^0 \arrow[ru]\arrow[d] \arrow[from=uu, crossing over]&\\
\Mptst_1 \times \Mptst_2 \arrow[rr] & & \Mptst. &
\end{tikzcd}
\end{equation}
\begin{lemma}
\label{lem:bigdrpbsquare}
All the squares in \ref{eq:bigdrdiagnonlog} are pullback squares.
\end{lemma}
\begin{proof}
Let $C_1/S, C_2/S$ be two log $1$-pointed curves. Let $C/S$ denote the gluing of these two curves, and let $f: (C_1 \sqcup C_2)^\punc \to C$ be the natural map. Let $\Lcal/C$ be a line bundle. Then the line bundle $\Lcal$ is trivial if and only if $f^* \Lcal$ is trivial, $\Lcal$ has multidegree $0$ if and only if $f^*\Lcal$ has multidegree $0$, and $\Lcal$ is twistable by a PL function to multidegree $0$ if and only if $f^* \Lcal$ is. The lemma immediately follows.
\end{proof}

In order to prove compatibility of $\LogDR$ with pullback along the gluing maps, we now need to find blowups of $\Mptst, \Mptst_1, \Mptst_2, \Divpt^0, \Divpt_1^0, \Divpt_2^0$ that are compatible with the gluing maps.




\begin{proposition}
\label{prop:compatiblesubdivions}
There are smooth log blowups $\widetilde{\Mpt}^\st, \widetilde{\Mpt}^\st_1, \widetilde{\Mpt}^\st_2$ of respectively $\Mptst, \Mptst_1, \Mptst_2$ that contain as opens smooth log blowups of respectively $\Divpt^0, \Divpt^0_1, \Divpt^0_2$ and that satisfy the condition \[\gl^* \widetilde{\Mpt}^\st = \widetilde{\Mpt}^\st_1 \times \widetilde{\Mpt}^\st_2.\]
\end{proposition}
\begin{proof}
We let $\Mpt^\trop, \Mpt^\trop_1, \Mpt^\trop_2$ denote the tropicalisations, i.e. the cone stacks in the sense of \cite{Cavalieri2020A-Moduli-Stack}, of $\Mptst_{g,n}, \Mptst_{g_1,n_1+1}, \Mptst_{g_2,n_2+1}$ respectively. Then there is a tropical gluing map $\Mpt^\trop_1 \times \Mpt^\trop_2 \to \Mpt^\trop$.

Similarly, we let $\Divpt^\trop \subset \Mpt^\trop, \Divpt^\trop_1 \subset \Mpt^{\trop}_1, \Divpt^\trop_2 \subset \Mpt^\trop_2$ denote the tropicalisations of the maps $\Divpt^0 \to \Mptst, \Divpt_1^0 \to \Mptst_1, \Divpt_2^0 \to \Mptst_2$.

Let $\Gamma$ be the decorated graph with two vertices and one edge, where the two vertices have genera $g_1$ and $g_2$, and carry $n_1$ resp. $n_2$ markings.

Now we pick two smooth subdivisions $\Sigma_1, \Sigma_2$ of $\Mpt^\trop_1, \Mpt^\trop_2$ that arise by pullback from smooth subdivisions of $\Mbar^\trop_1, \Mbar^\trop_2$ and that contain smooth subdivisions of $\Divpt^\trop_1, \Divpt^\trop_2$, and such that the corresponding subdivision $\Sigma_1 \times \Sigma_2$ of $\Mpt^\trop_1 \times \Mpt^\trop_2$ is $\Aut(\Gamma)$-invariant.

We let $\Mpt_{\Gamma}^\trop \subset \Mpt^\trop$ denote the image of $\Mpt^\trop_1 \times \Mpt^\trop_2$ under the gluing map. As the subdivision $\Sigma_1 \times \Sigma_2$ is $\Aut(\Gamma)$-invariant, it descends to a smooth subdivision $\Sigma_{\Gamma}$ of $\Mpt_{\Gamma}^\trop$. As $\Mpt_{\Gamma}^\trop$ is a maximal subcone stack of $\Mpt^\trop$, we can extend this to a subdivision $\Sigma$ of $\Mpt^\trop$ with $\Sigma|_{\Mpt_{\Gamma}^\trop} = \Sigma_{\Gamma}$.

As a tropical analogue of \ref{lem:bigdrpbsquare} we have the following pullback square of cone stacks
\begin{equation}
\begin{tikzcd}
\Divpt^\trop_1 \times \Divpt^\trop_2 \arrow[r]\arrow[d] &  \Divpt^\trop \arrow[d]\\
\Mpt^\trop_1 \times \Mpt^\trop_2 \arrow[r] & \Mpt^\trop. 
\end{tikzcd}
\end{equation}
This means $\Sigma$ already contains a smooth subdivision of $\Divpt^\trop|_{\Mpt_{\Gamma}^\trop}$. By repeated star subdivision in maximal cones of $\Sigma$ that do not lie in $\Mpt_{\Gamma}^\trop$, we obtain a smooth subdivision $\tilde{\Sigma}$ that contains a smooth subdivision of $\Divpt^\trop$ and such that $\tilde{\Sigma}|_{\Mpt_{\Gamma}^\trop} = \Sigma_{\Gamma}$.

Let $\widetilde{\Mpt}^\st, \widetilde{\Mpt}^\st_1, \widetilde{\Mpt}^\st_2$ be the log blowups obtained from the respective subdivisions $\tilde{\Sigma}, \Sigma_1, \Sigma_2$. Then by construction these satisfy all the requirements.
\end{proof}

\begin{remark}
Consider a non-degenerate small stability condition $\theta$, as defined in \cite{Kass2017The-stability-s}. It is shown in \cite[Theorem~A]{Holmes2022Logarithmic-double} that $\LogDR_g(\ul a)$ can be supported on a specific log blowup $\Mbar_{g,\ul a}^{\theta}$ depending on $\ul a$ and $\theta$. In \cite{Molcho2023SmoothCompactifications} Molcho constructs explicitly a smooth log blowup $\tilde{\ca M}_{g,\ul a}^\theta$ of $\Mbar_{g,\ul a}^{\theta}$. Choosing $\widetilde{\Mpt}^\st = \tilde{\ca M}_{g,\ul a}^\theta\times_{\Mbar_{g,n}}\Mptst_{g,n}$ in the above argument also works.
\end{remark}

\begin{proof}[Proof of \Cref{thm:drpartialcohft}]
We need to show that conditions~\oref{def:logcft:sninvariance}, \oref{def:logcft:separating}, \oref{def:logcft:unit1} and \oref{def:logcft:unit2} of \Cref{def:logcft} all hold.

The $S_n$-invariance \oref{def:logcft:sninvariance} and the unit axioms \oref{def:logcft:unit1} and \oref{def:logcft:unit2} are trivial. It remains to check the compatibility with pullback along the separating gluing map. We continue with the notation set up in this section.

We consider the smooth log blowups $\widetilde{\Mpt}^\st, \widetilde{\Mpt}^\st_1, \widetilde{\Mpt}^\st_2$ coming from \Cref{prop:compatiblesubdivions}. By combining \ref{eq:drdiaglog} and \ref{eq:bigdrdiagnonlog} we get the commutative diagram

\begin{equation}
\label{eq:bigdrdiaglog}
\begin{tikzcd}
 & \Mbar_1 \times \Mbar_2 \arrow[rr] \arrow[dd, "e_1 \times e_2", near start]&  & \Mbar \arrow[dd, "e", near start] \\
 Z_1 \times Z_2 \arrow[ru] \arrow[rr, crossing over] &  & Z \arrow[ru] &  \\
 & J_1 \times J_2 \arrow[rr] & & J \\
 \tilde{\Divpt}_1 \times \tilde{\Divpt}_2 \arrow[ru]\arrow[from=uu, crossing over]\arrow[d] \arrow[rr] &  & \tilde{\Divpt} \arrow[ru]\arrow[d] \arrow[from=uu, crossing over]&\\
 \tilde{\Mpt}^\st_1 \times \tilde{\Mpt}^\st_2 \arrow[rr, "\tilde{\gl}"', near start] & & \tilde{\Mpt}^\st  &
\end{tikzcd}
\end{equation}
where all the squares are pullback squares. We let $j_1 \times j_2$ denote the composite \[Z_1 \times Z_2 \to \tilde{\Mpt}^{\st}_1 \times \tilde{\Mpt}^{\st}_2, \]
 and $j$ the composite \[Z \to \tilde{\Mpt}^\st. \]

We now perform the following computation inside $\LogCH(\Mptst_1 \times \Mptst_2)$.
\begin{align*}
\gl^* \Omega_{g,n}(\ul a) =& \tilde{\gl}^* (\LLogDR_{g,n}(\ul a)) & \\
 =& \tilde{\gl}^* (j_* e^!([\tilde\Divpt]))^{\text{PD}} & \text{by \Cref{lem:llogdr}}\\
 =& \left(\tilde{\gl}^! j_* e^!([\tilde\Divpt])\right)^{\text{PD}} & \\
 =& \left((j_1 \times j_2)_* \tilde{\gl}^! e^! ([\tilde\Divpt])\right)^{\text{PD}} & \\
 =& \left((j_1 \times j_2)_* e^! \tilde{\gl}^! ([\tilde\Divpt])\right)^{\text{PD}} & \\
 =& \left((j_1 \times j_2)_* (e_1 \times e_2)^! \tilde{\gl}^! ([\tilde\Divpt])\right)^{\text{PD}} & \\
 =& \left((j_1 \times j_2)_* (e_1 \times e_2)^! ([\tilde\Divpt_1 \times \tilde\Divpt_2])\right)^{\text{PD}} & \\
 =& (j_{1,*} e_1^! [\tilde\Divpt_1])^{\text{PD}} \boxtimes (j_{2,*} e_2^! [\tilde\Divpt_2])^{\text{PD}} & \\
 =& \LLogDR_{g_1,n_1+1}(a_1,\dots,a_{n_1}, -\sum_{i=1}^{n_1} a_i) & \\
 & \boxtimes \LLogDR_{g_2,n_2+1}(a_{n_1 + 1}, \dots, a_{n}, \sum_{i=1}^{n_1} a_i) & \text{by \Cref{lem:llogdr}}
 \end{align*}
This was exactly what we needed to check for condition \oref{def:logcft:separating}. We conclude that $\Omega$ satisfies conditions~\oref{def:logcft:sninvariance}, \oref{def:logcft:separating}, \oref{def:logcft:unit1} and \oref{def:logcft:unit2} of \Cref{def:logcft}.
\end{proof}

\appendix

\section{Comparison with log Chow rings of Barrott}
\label{sec:logchow}

A general theory of bivariant log Chow rings is under development by Barrott, in \cite{Barrott2019Logarithmic-Cho}. In this section we show that the log Chow ring we defined in \ref{def:logchow} has a map to the log Chow ring of Barrott (under some mild conditions on the space), and thereby show that all our results still hold for the log Chow rings defined by Barrott. We first give a short summary of Barrott's definition.

If $f\colon X \to Y$ is a log blowup of log schemes with source and target locally free, Barrott defines a Gysin map $f^!\colon \CH_*(Y) \to \CH_*(X)$. For $X$ a log scheme, he then defines the log Chow group of $X$ as 
\begin{equation}
\CH^\dagger_*(X) = \colim_{\tilde X \to X} \CH_*(\tilde X)
\end{equation}
where the colimit is taken over locally free log blowups of $X$, with transition maps given by the Gysin pullback. For $X \to Y$ a map of log schemes, he proposes definitions of various operations, including: 
\begin{enumerate}
\item 
if $f\colon X \to Y$ is proper, a pushforward map $f_*\colon \CH^\dagger_*(X) \to \CH^\dagger_*(Y)$; 
\item
if $f\colon X \to Y$ is log flat, a pullback map $f^*\colon \CH^\dagger_*(Y) \to \CH^\dagger_*(X)$; 
\item 
if $f\colon X \to Y$ is a strict regular embedding, a Gysin pullback map $f^!\colon \CH^\dagger_*(Y) \to \CH^\dagger_*(X)$; 
\item 
if $f\colon X \to Y$ is a log blowup with source and target locally free, a Gysin pullback map $f^!\colon \CH^\dagger_*(Y) \to \CH^\dagger_*(X)$. 
\end{enumerate}

Barrott then defines the \emph{log Chow cohomology} ring $\CH_\dagger^*(X)$ of a log scheme $X$; an object $x \in \CH_\dagger^*(X)$ consists of the data of, for every log scheme $T \to X$, a morphism $z_T\colon \CH^\dagger_*(T) \to \CH^\dagger_*(T)$, and these maps $z_T$ should commute with saturated proper pushforward, log flat pullback, and strict Gysin pullback. 

\begin{lemma}\label{lem:op_compatibilities}
Let $X$ be a log scheme locally of finite type over $k$. Let $z \in \CHop(X)$ and let $t\colon T \to X$ be a morphism of log schemes. Then 
\begin{enumerate}
\item The morphisms $z_{\tilde T}\colon \CH_*(\tilde T)  \to \CH_*(\tilde T) $ for $\tilde T \to T$ log blowups assemble into a group homomorphism 
\begin{equation}
z^\dagger_T\colon \CH^\dagger_*(T) \to \CH^\dagger_*(T). 
\end{equation}
\item The data of the maps $z_T^\dagger$ for varying $T$ from (1) commute with saturated proper pushforward, log flat pullback, and strict Gysin pullback, and hence define an element $z^\dagger \in \CH_\dagger^*(X)$. 
\end{enumerate}
\end{lemma}
\begin{proof}
The first claim is \ref{lem:key_Gysin_compatibility} below. For the second claim, commutation with strict Gysin pullback is immediate from \cite[Def 17.1, C1]{Fulton1984Intersection-th}. Commutation with log flat pullback and saturated proper pushforward are more involved. After unravelling Barrott's definitions, all the constructions are composed of proper pushforward, flat pullback, and Barrott's Gysin pullback. Compatibility with the former two operations holds by definition of the operational class \cite[Def 17.1]{Fulton1984Intersection-th}, so it remains to check compatibility with Barrott's Gysin pullback, which is \ref{lem:key_Gysin_compatibility}.
\end{proof}
\begin{lemma}\label{lem:key_Gysin_compatibility}
Let $t \colon T \to X$ be a morphism from a locally free finite-type log scheme, and let $\pi\colon \tilde T \to T$ be a log blowup with $
\tilde T$ also locally free. Let $z \in \CHop(X)$. Then the diagram 
\begin{equation}
 \begin{tikzcd}
  \CH_*(T) \arrow[d, "z_T"] \arrow[r, "\pi^!"] & \CH_*(\tilde T) \arrow[d, "z_{\tilde T}"]\\
    \CH_*(T)  \arrow[r, "\pi^!"] & \CH_*(\tilde T) 
\end{tikzcd}
\end{equation}
commutes, where $\pi^!$ is Barrott's Gysin pullback. 
\end{lemma}
\begin{proof}
Suppose first that we can find a cartesian square 
\begin{equation}
 \begin{tikzcd}
  \tilde T \arrow[r]\arrow[d, "\pi"] & \tilde {\ca T} \arrow[d, "\Pi"]\\
T \arrow[r] & \ca T 
\end{tikzcd}
\end{equation}
with horizontal arrows strict, and $\Pi$ a log blowup between smooth log smooth log schemes (in particular, $\Pi$ is l.c.i.). Then Barrott's $\pi^!$ is just the Gysin pullback induced by $\Pi$, and the result is immediate from the commutation of $z$ with smooth pullbacks and with Gysin pullback along regular embeddings. 

In general such a cartesian square need not (we suppose) exist. However, it's role can be played by the Artin fans $\ca A_T$ and  $\ca A_{\tilde T}$ of $T$ and $\tilde T$. We have a cartesian diagram 
\begin{equation}
 \begin{tikzcd}
  \tilde T \arrow[r]\arrow[d, "\pi"] & \ca A_{\tilde T} \arrow[d, "\Pi"]\\
T \arrow[r] & \ca A_T 
\end{tikzcd}
\end{equation}
where $\Pi$ is a representable l.c.i. map between smooth log smooth log algebraic stacks. Following \cite[\S 3.1]{Kresch1999Cycle-groups-fo} we have a refined Gysin pullback $\Pi^!$, and this coincides with Barrott's pullback. 

Now, the operational class $z$ commutes (by definition) with Gysin pullbacks for lci morphisms of schemes, but not a-priori lci morphisms of stacks. However, this is in fact automatic. For smooth morphisms this is clear, so it is enough to check compatibility with Gysin pullbacks along regular closed immersions of stacks. The precise statement is \ref{lem:stacky_Gysin_comp}. 
%
%
\end{proof}

\begin{lemma}\label{lem:stacky_Gysin_comp}
Let $X$ be a scheme of finite type over $k$, and let $z \in \CHop(X)$ be an operational class. Let $f\colon S' \to S$ be a regular closed immersion of algebraic stacks of finite type over $k$. Let $T$ be a finite-type $k$ scheme with maps to $X$ and to $G$. Consider the fibre diagram 
\begin{equation}
 \begin{tikzcd}
T' \arrow[r]\ar[d] & S' \ar[d, "f"]\\
T \ar[r] \ar[d]& S\\
X. 
\end{tikzcd}
\end{equation}
Then the diagram 
\begin{equation}
 \begin{tikzcd}
  \CH_*(T) \arrow[d, "z_T"] \arrow[r, "f^!"] & \CH_*(T') \arrow[d, "z_{T'}"]\\
    \CH_*(T)  \arrow[r, "f^!"] & \CH_*(T') 
\end{tikzcd}
\end{equation}
commutes, where $f^!$ is Kresch's Gysin pullback. 
\end{lemma}
\begin{proof}
The `stacky Gysin pullback' $f^!$ is constructed in \cite[\S 3.1]{Kresch1999Cycle-groups-fo} as a composite of several maps. In our situation each step in the construction is either a proper pushforward, flat pullback, or Gysin pullback along a map of schemes. In particular, these all commute with the action of $z$. 
\end{proof}
Having finished proving \ref{lem:op_compatibilities}, we collect some consequences. First, \ref{lem:op_compatibilities} yields a ring homomorphism 
\begin{equation}\label{lem:cohom_to_barr}
\CHop(X) \to \CH_\dagger^*(X). 
\end{equation}

\begin{lemma}
If $\pi\colon \tilde X \to X$ is a log blowup, there is a canonical isomorphism $\CH_\dagger^*(\tilde X) = \CH_\dagger^*(X)$, and the diagram 
\begin{equation}
 \begin{tikzcd}
  \CHOP(X) \arrow[r] \arrow[d]&  \CHop(X) \arrow[r] \arrow[d]& \CH_\dagger^*( X)\arrow[d] \\
    \CHOP(\tilde  X) \arrow[r] &    \CHop(\tilde  X) \arrow[r] & \CH_\dagger^*(\tilde X)
\end{tikzcd}
\end{equation}
commutes, where the right horizontal arrows are from \ref{lem:cohom_to_barr}, and the middle and left vertical arrows are the pullback of Chow cohomology. 
\end{lemma}
\begin{proof}
Immediate from the definition of the pullback on Chow cohomology and the construction of the map \ref{lem:cohom_to_barr}. 
\end{proof}

In this way we build the promised ring homomorphism
\begin{equation}
\label{eq:maptobarrot}
\LogCH(X) \to \CH^*_\dagger(X). 
\end{equation}

The following lemma on compatibility of pullbacks follows immediately.
\begin{lemma}
\label{lem:pullbackcommutebarrot}
Let $f\colon X \to Y$ be a map of dominable log stacks. Then there is a commutative square
\[
 \begin{tikzcd}
  \LogCH(Y) \arrow[r, "f^*"] \arrow[d]& \LogCH(X)\arrow[d] \\
    \CH^*_\dagger(Y) \arrow[r, "f^*"] & \CH^*_\dagger(X)
\end{tikzcd}
\]
\end{lemma}

We do not know whether the map \ref{eq:maptobarrot} is injective or surjective. In this paper we work with the naive ring $\LogCH(X)$, but with this map and the compatibility lemma \oref{lem:pullbackcommutebarrot} our constructions and results transfer immediately to Barrott's setup of log Chow cohomology rings $\CH^*_\dagger(X)$. 

\section{Tropical gluing}
\label{sec:tropical}

In this section we will introduce \emph{tropical pointed curves}. These are the tropical versions of log pointed curves, and many of the properties enjoyed by log pointed curves have a tropical analogue. For example, there are tropical gluing maps and tropical evaluation maps. In this section we omit all proofs, and refer the interested reader to the logarithmic versions of the statements. 

We start by recalling the notion of a tropical unpointed curve. For more details we refer to \cite[Section~3.1]{Cavalieri2020A-Moduli-Stack}.

\begin{definition}
We denote by $G_{g,n}$ the set of graphs
\[
G = (V, H, L = (p_i)_{i = 1}^n \subset H, r: H \to V, i: H \to H, g: V \to \N)
\]
where
\begin{enumerate}
  \item $V$ is the set of vertices;
  \item $H$ is the set of half edges;
  \item $i$ is an involution on $H$.
  \item $r$ assigns to every half edge the vertex it is incident to. 
  \item $L \subset H$ is a list of the legs, the fixed points of the involution $i$, required to be of size $n$;
  \item $(V,H)$ is a connected graph;
  \item $g(v)$ is the genus of vertex $v$;
  \item the graph is stable, i.e. for every vertex $v$ we have $2g(v)-2 + n(v) > 0$ where $n(v)$ is the number of half edges incident to $v$;
  \item the total genus $\sum_{v \in V} g(v) + h^1(G)$ is $g$;
\end{enumerate}
We call this a genus $g$ graph with $n$ markings.
\end{definition}

Recall the category of rational polyhedral cones $\RPC$ and the category of rational polyhedral cone complexes $\RPCC$ from \cite[Section~2.1]{Cavalieri2020A-Moduli-Stack}. 

Now we can define the notion of a pointed tropical curve as in \cite{Cavalieri2020A-Moduli-Stack}.

\begin{definition}
Let $\sigma \in \RPC$ and let $P$ be the corresponding sharp dual monoid. A pointed tropical curve $\Gamma/\sigma$ of genus $g$ with $n$ markings consists of a graph $G \in G_{g,n}$ and a length function $\ell: H \setminus L \to P_{>0}$ invariant under the involution $i$. 
\end{definition}

This notion then automatically extends to a category over $\RPCC$, fibred in groupoids. This category is called $\Mbar_{g,n}^\trop$, and it is shown in \cite{Cavalieri2020A-Moduli-Stack} that this is a cone stack and has many similarities to the moduli stack of log curves.

Similar to the moduli stack of log curves, there is no notion of gluing for pointed tropical curves. However, as for log curves, one can define a notion of tropical pointed curves for which gluing does exist.

\begin{definition}
\label{def:tropicalpointed}
Let $\sigma \in \RPC$ and let $P$ be the corresponding sharp dual monoid. A tropical pointed curve $\Gamma/\sigma$ of genus $g$ with $n$ markings consists of a graph $G \in J_{g,n}$ and a $\ell: H \to P_{>0}$ invariant under the involution $i$. 
\end{definition}

This again immediately extends to a category over $\RPCC$ fibred in groupoids.

\begin{definition}
We let $\Mpt_{g,n}^\trop/\RPCC$ denote the cone stack of tropical pointed curves.
\end{definition}

\begin{remark}
Under the equivalence of cone stacks and Artin fans, $\Mpt_{g,n}^\trop$ is the Artin fan of $\Mptst_{g,n}$.
\end{remark}

\begin{remark}
Compare \Cref{def:tropicalpointed} with \Cref{def:ellpi}, where it is shown that for a log pointed curve the legs naturally have lengths whose values lie in the characteristic monoid of the base. In fact, imitating \cite[Section~7]{Cavalieri2020A-Moduli-Stack}, a stable log pointed curve over an algebraically closed field can be tropicalised to obtain a tropical pointed curve.
\end{remark}

\begin{proposition}
\label{prop:trop:pointedtimesrn}
There is a natural isomorphism of cone stacks \[\Mpt_{g,n}^\trop \cong \Mbar_{g,n}^\trop \times \R_{\geq 0}^n.\]
\end{proposition}

From now on we use tropical curve to mean tropical pointed curve. We can now see that we can glue tropical curves, as per the following definition.

\begin{definition}[Cf. \Cref{sec:sewing}]
\label{def:trop:gluing}
Let $\Gamma/\sigma$ be a (not necessarily connected) tropical curve and let $p_1, p_2 \in H$ be two legs. Let $G$ denote the graph of $\Gamma$ and let $G^\gl$ denote the graph $G$ where $p_1, p_2$ have been removed and been replaced by two half-edges $h_1,h_2$ with $i(h_1) = h_2$. We define the \emph{gluing} $\Gamma^\gl$ to be the tropical curve with underlying graph $G^\gl$ and length function $\ell$ given by $\ell(h_1) = \ell(h_2) = \ell(p_1) + \ell(p_2)$, and all other lengths equal to those of $\Gamma$.
\end{definition}

By applying this to either the disjoint union of two tropical curves, or to a single tropical curve, we get the following tropical gluing maps.

\begin{definition}[Cf. \Cref{sec:sewing:sep}]
The gluing construction \Cref{def:trop:gluing} defines natural gluing maps
\[
\Mpt_{g_1,n_1+1}^\trop \times \Mpt_{g_2,n_2+1}^\trop \to \Mpt_{g_1+g_2,n_1+n_2}^\trop
\]
and 
\[
\Mpt_{g-1,n+2}^\trop \to \Mpt_{g,n}^\trop.
\]
\end{definition}

Now we will discuss the other benefits of tropical pointed curves, namely evaluation maps and the ability to glue maps. To do this, we first recall that any unpointed tropical curve $\Gamma/\sigma$ can be realised as a relative cone complex of relative dimension $1$. One can realise pointed tropical curves either by realising the legs as having length $0$ or length $\infty$, both with their advantages and disadvantages. For tropical pointed curves, we have the following notion of realisation.

\begin{definition}
Let $\Gamma/\sigma$ be a tropical curve with $n$ marked points. We write $\Gamma_{\R}/\sigma$ for the relative cone complex of relative dimension $1$ containing one copy of $\sigma$ for every vertex, and for every half leg $h$ up to the involution, one cone of relative dimension $1$ over $\sigma$ such that the fiber over $x \in \sigma$ has length $\ell(h)(x)$. For $i = 1, \dots, n$ we write $p_i\colon \sigma \to \Gamma_{\R}$ for the map that over every fiber, sends $s \in \sigma$ to the endpoint of leg $i$ in $\Gamma_{\R}|_{s}$.
\end{definition}

Now we can state the existence of evaluation maps and the universal property of the gluing.

\begin{definition}[Cf. \Cref{subsec:evaluation}]
Let $X$ be a cone complex. Let $\Gamma/\sigma$ be a tropical curve, and let $p_1, \dots, p_n$ be the legs. For $X$ a cone complex and $f\colon \Gamma_{\R} \to X$ a map of cone complexes, we write $f(p_i)$ for the composition $f \circ p_i\colon \sigma \to X$.
\end{definition}

\begin{theorem}[Cf. \Cref{thm:universalpropertygluing}]
Let $X$ be a cone complex. Let $\Gamma/\sigma$ be a (not necessarily connected) tropical curve, and let $p_1, p_2$ be two legs. Then there is a natural bijection
\begin{equation}
\Hom(\Gamma_{\R}^{\gl}, X) \isom \left\{f \in \Hom(\Gamma_{\R}, X) : \begin{aligned}&f(p_1) = f(p_2) \text{ and the slopes} \\& \text{of } f \text{ along } p_1, p_2 \text{ add up to } 0\end{aligned}\right\}.
\end{equation}
\end{theorem}

\bibliographystyle{alpha} 
\bibliography{../prebib.bib}

\end{document}